\documentclass[a4paper,reqno,10pt]{amsart}

\usepackage{amssymb}
\usepackage{amstext}
\usepackage{amsmath}
\usepackage{amscd}
\usepackage{amsthm}
\usepackage{amsfonts}
\usepackage{enumerate}
\usepackage{graphicx}
\usepackage{latexsym}
\usepackage{mathrsfs}
\usepackage{mathtools}
\usepackage[all,poly,necula]{xy}
\usepackage{ifthen}
\xyoption{all}

\usepackage{lscape}
\usepackage{ stmaryrd }
\usepackage{multirow}

\newtheorem{theorem}{Theorem}[section]
\newtheorem{theoremi}{Theorem}
\newtheorem{corollaryi}[theoremi]{Corollary}

\newtheorem{corollary}[theorem]{Corollary}
\newtheorem{lemma}[theorem]{Lemma}
\newtheorem{proposition}[theorem]{Proposition}
\newtheorem{definition-proposition}[theorem]{Definition-Proposition}

\theoremstyle{definition}
\newtheorem{definition}[theorem]{Definition}

\newtheorem{remark}[theorem]{Remark}
\newtheorem{example}[theorem]{Example}
\newtheorem*{examplei}{Example}

\renewcommand{\AA}{\mathcal{A}}
\newcommand{\CC}{\mathcal{C}}

\newcommand{\DD}{\mathcal{D}}

\newcommand{\FF}{\mathcal{F}}
\newcommand{\PP}{\mathcal{P}}
\newcommand{\II}{\mathcal{I}}
\newcommand{\OO}{\mathcal{O}}

\renewcommand{\SS}{\mathcal{S}}

\newcommand{\UU}{\mathcal{U}}
\newcommand{\XX}{\mathcal{X}}
\newcommand{\YY}{\mathcal{Y}}

\newcommand{\ig}{\mathbf{i}}
\newcommand{\kg}{\mathbf{k}}
\newcommand{\tg}{\mathbf{t}}

\renewcommand{\P}{\mathbb{P}}

\newcommand{\N}{\mathbb{N}}

\newcommand{\length}{\operatorname{length}\nolimits}

\newcommand{\injdim}{\operatorname{inj.\!dim}\nolimits}
\renewcommand{\top}{\operatorname{top}\nolimits}
\newcommand{\cotop}{\operatorname{cotop}\nolimits}
\newcommand{\Bcotop}{B\!\operatorname{-cotop}\nolimits}
\newcommand{\Bcorad}{B\!\operatorname{-corad}\nolimits}
\newcommand{\Bpcotop}{B'\!\!\operatorname{-cotop}\nolimits}

\newcommand{\soc}{\operatorname{soc}\nolimits}
\newcommand{\rad}{\operatorname{rad}\nolimits}
\newcommand{\corad}{\operatorname{corad}\nolimits}
\newcommand{\Ext}{\operatorname{Ext}\nolimits}

\newcommand{\Tor}{\operatorname{Tor}\nolimits}
\newcommand{\Hom}{\operatorname{Hom}\nolimits}

\newcommand{\End}{\operatorname{End}\nolimits}

\newcommand{\gl}{\operatorname{gl.\!dim}\nolimits}
\newcommand{\op}{\operatorname{op}\nolimits}
\newcommand{\RHom}{\mathbf{R}\strut\kern-.2em\operatorname{Hom}\nolimits}

\newcommand{\Kernel}{\operatorname{Ker}\nolimits}
\newcommand{\Cokernel}{\operatorname{Coker}\nolimits}

\newcommand{\SL}{\operatorname{SL}\nolimits}
\newcommand{\Ab}{\mathcal{A}b}
\newcommand{\coker}{\Cokernel}

\renewcommand{\ker}{\Kernel}
\def\pup*{\textsuperscript{+}}

\DeclareMathOperator{\moduleCategory}{\mathsf{mod}} \renewcommand{\mod}{\moduleCategory}

\DeclareMathOperator{\proj}{\mathsf{proj}}

\DeclareMathOperator{\ind}{\mathsf{ind}}
\DeclareMathOperator{\Sub}{\mathsf{Sub}}

\DeclareMathOperator{\Gr}{\mathsf{Gr}}

\DeclareMathOperator{\CM}{\mathsf{CM}}

\DeclareMathOperator{\add}{\mathsf{add}}
\DeclareMathOperator{\fl}{\mathsf{f.\!l.}}

\def\innew{0}
\newcommand{\new}[1]{\if\innew1{\cyan {#1}}\else{\def\innew{1}\blue {#1}\def\innew{0}}\fi}

\newcommand{\cut}{\ar@{-}@[|(5)]}

\newcommand{\rs}[1]{\setcounter{enumi}{#1}}

\newenvironment{sbmatrix}{\left[\begin{smallmatrix}}{\end{smallmatrix}\right]}

\newsavebox\locboxinminipage
\newlength\locboxinminipagel
\newcommand{\boxinminipage}[1]
{%
 \sbox\locboxinminipage{#1}%
 \settowidth\locboxinminipagel{\usebox{\locboxinminipage}}%
 \begin{minipage}{\locboxinminipagel}\usebox{\locboxinminipage}\end{minipage}%
}

\def\newboxedcommand#1#2 
{%
 \def\newboxedcommandlocala##1##2.{##2}%
 \edef\newboxedcommandlocalb{\expandafter\newboxedcommandlocala\string#1.}%
 \expandafter\newsavebox\csname\newboxedcommandlocalb savebox\endcsname%
 \expandafter\sbox\csname\newboxedcommandlocalb savebox\endcsname{#2}%
 \expandafter\newlength\csname\newboxedcommandlocalb largeurbox\endcsname%
 \expandafter\settowidth\csname\newboxedcommandlocalb largeurbox\endcsname{\usebox{\csname\newboxedcommandlocalb savebox\endcsname}}%
 \edef#1{\noexpand\begin{minipage}{\csname\newboxedcommandlocalb largeurbox\endcsname}\usebox{\csname\newboxedcommandlocalb savebox\endcsname}\noexpand\end{minipage}}%
}

\newcommand{\Cf}{\mathbb{C}}
\renewcommand{\AA}{\mathcal{A}}
\newcommand{\BB}{\mathcal{B}}
\newcommand{\EE}{\mathcal{E}}
\DeclareMathOperator{\id}{id}

\numberwithin{equation}{section}

\xymatrixrowsep{1.5em}

\begin{document}

\title[Lifting algebras to orders and categorification]{Lifting preprojective algebras to orders and categorifying partial flag varieties}

\author[Laurent Demonet]{Laurent Demonet}
\address{L. Demonet: Graduate School of Mathematics, Nagoya University, Furocho, Chikusaku, Nagoya 464-8602, Japan}
\email{Laurent.Demonet@normalesup.org}
\urladdr{http://www.math.nagoya-u.ac.jp/~demonet/}
\thanks{The first named author was partially supported by JSPS Grant-in-Aid for Young Scientist (B) 26800008.}

\author[Osamu Iyama]{Osamu Iyama}
\address{O. Iyama: Graduate School of Mathematics, Nagoya University, Furocho, Chikusaku, Nagoya 464-8602, Japan}
\email{iyama@math.nagoya-u.ac.jp}
\urladdr{http://www.math.nagoya-u.ac.jp/~iyama/}
\thanks{The second named author was partially supported by JSPS Grant-in-Aid for Scientific Research (B) 24340004, (B) 16H03923, (C) 23540045, and (S) 15H05738.}


\begin{abstract}
 We describe a categorification of the cluster algebra structure of multi-homogeneous  coordinate rings
 of partial flag varieties of arbitrary Dynkin type using Cohen-Macaulay modules over orders. This completes the categorification of Geiss-Leclerc-Schr\"oer by adding the missing coefficients. To achieve this, for an order $A$ and an idempotent $e \in A$, we introduce a subcategory $\CM_e A$ of $\CM A$ and study its properties. In particular, under some mild assumptions, we construct an equivalence of exact categories $(\CM_e A)/[Ae] \cong \Sub Q$ for an injective $B$-module $Q$ where $B := A/(e)$. These results generalize work by Jensen-King-Su concerning the cluster algebra structure of the Grassmannian $\Gr_m(\Cf^n)$.
\end{abstract}

\maketitle

\tableofcontents

\section{Introduction}

In \cite{GeLeSc08}, Geiss-Leclerc-Schr\"oer introduced a cluster algebra structure on some subalgebra $\tilde \AA$ of the multi-homogeneous coordinate ring $\Cf[\FF]$ of the partial flag variety $\FF=\FF(\Delta,J)$ corresponding to a Dynkin diagram $\Delta$ and a set $J$ of vertices of $\Delta$. They prove that $\tilde \AA = \Cf[\FF]$ in type $A$, and conjecture that the equality holds after an appropriate localization for any Dynkin type (see Section \ref{parflag} for more details). This structure generalizes previously known cases of Grassmannians, introduced from the beginning for $\Gr_2(\Cf^n)$ by Fomin and Zelevinsky \cite{FoZe03} (see also \cite{BeFoZe05}) and generalized by Scott for $\Gr_k(\Cf^n)$ \cite{Sc06}.

In the same paper, Geiss-Leclerc-Schr\"oer introduce a partial categorification of this cluster algebra structure on $\tilde \AA$. A crucial role is played by the preprojective algebra $\Pi$ of type $\Delta$ and a certain full subcategory $\Sub Q_J$ of $\mod \Pi$ which is Frobenius and stably $2$-Calabi-Yau. More precisely, they introduce a cluster character $\tilde \varphi: \Sub Q_J \rightarrow \tilde \AA$ which gives a bijection
 \begin{align*}& \{\text{reachable indecomposable rigid objects in $\Sub Q_J$} \}/\cong \\ \xleftrightarrow{1-1}\, & \{\text{cluster variables and coefficients of } \tilde \AA \} \setminus \{\Delta_j \mid j \in J\},\end{align*}
where, for $j \in J$, $\Delta_j$ is the corresponding principal generalized minor.

One of the aim of this paper is to look for a stably $2$-Calabi-Yau category extending $\Sub Q_J$ whose reachable indecomposable rigid objects correspond to cluster variables and all coefficients of $\tilde \AA$.
In \cite{JeKiSu}, Jensen-King-Su achieved this in the case of classical Grassmannians (\emph{i.e.} $\Delta = A_n$ for $n \geq 0$ and $\# J = 1$) by using orders (see also \cite{BaKiMa} for an interpretation in terms of dimer models). In this article, we extend their method to any arbitrary Dynkin diagram $\Delta$ and arbitrary set of vertices $J$.

Throughout the introduction, for simplicity, let $R := k\llbracket t \rrbracket$ be the formal power series ring over an arbitrary field $k$. For an $R$-order $A$ (\emph{i.e.} an $R$-algebra that is free of finite rank as an $R$-module), we denote by $\CM A$ the category of Cohen-Macaulay modules over $A$ (\emph{i.e.} $A$-modules that are free of finite rank over $R$). For an idempotent $e \in A$, we define
 $$\CM_e A := \{X \in \CM A\mid eX \in \proj (eAe)\}.$$
We prove the following result:
\begin{theoremi}[Theorems \ref{modelAD} and \ref{categG}] \label{thparflag}
 Let $\Delta$ be a Dynkin diagram, and $J$ be a set of vertices of $\Delta$. Then, there exist a $\Cf\llbracket t \rrbracket$-order $A$, an idempotent $e \in A$ such that $\CM_e A$ is Frobenius and stably $2$-Calabi-Yau, and a cluster character $\psi: \CM_e A \rightarrow \tilde \AA$ such that
 \begin{enumerate}[\rm (a)]
   \item $\psi$ induces a bijection between
    \begin{itemize}
     \item isomorphism classes of reachable indecomposable rigid objects of $\CM_e A$;
     \item cluster variables and coefficients of $\tilde \AA$.
    \end{itemize}
   \item $\psi$ induces a bijection between
    \begin{itemize}
     \item isomorphism classes of reachable basic cluster tilting objects of $\CM_e A$;
     \item clusters of $\tilde \AA$.
    \end{itemize}
   Moreover, it commutes with mutation of cluster tilting objects and mutation of clusters.  
  \end{enumerate} 
\end{theoremi}

To prove Theorem \ref{thparflag}, we generalize techniques introduced by Jensen-King-Su \cite{JeKiSu} for Grassmannians in type $A$ (see also \cite{DeLu16} for Grassmannians of $2$-dimensional planes in type $A$). Meanwhile, we need to prove general results on orders.

The study of Cohen-Macaulay modules (also known as lattices) over orders is a classical subject in representation theory. We refer to \cite{Au78,CuRe81,LeWi12,Si92,Yo90} for a general background on this subject. We also refer to \cite{AmIyRe15,Ar99,DeLu16,DeLu16-2,HeIyMiOp,ThVa10,IyTa13,KaSaTa07,KaSaTa09,KeRe08} for recent results about connections with tilting theory and cluster categories.

We consider an $R$-order $A$ and an idempotent $e \in A$ such that $B := A/(e)$ is finite dimensional over $k$. Let $K := k(\!(t)\!)$ be the fraction field of $R$, $U := \Hom_A(B, Ae \otimes_R (K/R))$ and $\Sub U$ be the category of $B$-submodules of objects $U^n$ for $n \geq 0$. We consider the exact full subcategory $\mod_e A := \{X \in \mod A\mid eX \in \proj (eAe)\}$ of $\mod A$. Under this setting, we prove the following generalization of a result of \cite{JeKiSu}. 
\begin{theoremi}[Theorem \ref{mainB}] \label{tho}
 Assume that $Ae$ is injective in $\CM_e A$ and has injective dimension at most $1$ in $\mod_e A$. Then $U$ is injective in $\mod B$ and there is an equivalence of exact categories
 $$B \otimes_A -:(\CM_e A)/[Ae] \xrightarrow{\sim} \Sub U.$$
\end{theoremi}

In particular, if $e$ and $g$ are idempotents of an $R$-order $A$ such that $Ae \cong \Hom_R(gA, R)$ as left $A$-modules and $B = A/(e)$ is finite dimensional, then the hypotheses of Theorem \ref{tho} are satisfied and $U$ is the injective $B$-module corresponding to the idempotent $g$ (see Theorem \ref{mainA}). Let us give a motivating example:
\begin{examplei}
 For $n \geq 1$, we consider the pair $(A,e)$ defined as follows:
 $$A := \begin{bmatrix} R & R \\ (t^n) & R \end{bmatrix} \quad \text{and} \quad e := \begin{bmatrix} 1 & 0 \\ 0 & 0 \end{bmatrix}$$
 We have $Ae \cong \Hom_R((1-e) A, R)$ and $B = A / (e) \cong k[t]/(t^n)$. So according to Theorem \ref{tho}, $(\CM_e A)/[Ae] \cong \Sub U = \mod B$. Notice that here $\CM (e A e) = \proj (e A e)$ so $\CM_e A = \CM A$. We can illustrate this fact by drawing the Auslander-Reiten quivers of $\CM A$ and $\mod B$:
 $$
\xymatrix@R=1.5cm@C=2.2cm@!=0cm{
  \CM_e A : \ar[d]_{B \otimes_A -}   \\
  \mod B : 
 }
 \xymatrix@R=1.5cm@C=2.2cm@!=0cm{
   \boxinminipage{$\begin{bmatrix} R \\ (t^n) \end{bmatrix}$} \ar@<3pt>[r] & \boxinminipage{$\begin{bmatrix} R \\ (t^{n-1}) \end{bmatrix}$} \ar@<3pt>[l]^-{t} \ar@<3pt>[r] & \boxinminipage{$\begin{bmatrix} R \\ (t^{n-2}) \end{bmatrix}$} \ar@<3pt>[l]^-{t} \ar@<3pt>[r] & \cdots \ar@<3pt>[l]^-{t} \ar@<3pt>[r] & \boxinminipage{$\begin{bmatrix} R \\ (t) \end{bmatrix}$} \ar@<3pt>[l]^-{t} \ar@<3pt>[r] & \boxinminipage{$\begin{bmatrix} R \\ R \end{bmatrix}$} \ar@<3pt>[l]^-{t}  \\
   & k[t]/(t) \ar@<3pt>[r]^-{t} & k[t]/(t^2) \ar@<3pt>[l] \ar@<3pt>[r]^-{t} & \cdots \ar@<3pt>[l] \ar@<3pt>[r]^-{t} & k[t]/(t^{n-1}) \ar@<3pt>[l] \ar@<3pt>[r]^-{t} & k[t]/(t^n) \ar@<3pt>[l]
 }$$
 where projective-injective objects are leftmost and rightmost in the first row and only rightmost in the second row. On the other objects, the Auslander-Reiten translation acts as the identity.
\end{examplei}

 As an application of Theorem \ref{tho}, we get the following, which is fundamental for Theorem \ref{thparflag}:

 \begin{corollaryi}[Corollary of Theorem \ref{mainA}] \label{thselfinj}
  Let $B$ be a finite dimensional selfinjective $k$-algebra. We define a Gorenstein order $A$ over $R = k\llbracket t \rrbracket$ and an idempotent $e$ of $A$ by
  $$A := B \otimes_{k} \begin{bmatrix} R & R \\ tR & R\end{bmatrix} \quad \text{and} \quad e := \begin{bmatrix} 1 & 0 \\ 0 & 0 \end{bmatrix}.$$
  Then we have an equivalence of exact categories $(\CM_e A)/[Ae] \cong \mod B$ which induces a triangle equivalence $\underline{\CM}_e A \cong \underline{\mod} B$ between stable categories.
 \end{corollaryi}

We also prove a categorical version of Theorem \ref{tho} in the context of exact categories:

\begin{theoremi}[Theorem \ref{abelian}] \label{thc}
 Let $\EE$ be an exact category which is $\Hom$-finite over a field $k$. We suppose that
 \begin{itemize}
  \item $(\AA, \BB)$ and $(\BB, \CC)$ are torsion pairs in $\EE$;
  \item $\EE$ has enough projective objects, which belong to $\CC$;
  \item There exists a projective object $P$ in $\EE$ which is injective in $\CC$ and satisfies $\AA=\add P$;
  \item $\BB$ is an abelian category whose exact structure is compatible with that of $\EE$.
 \end{itemize}
 Then, there is an equivalence of exact categories
 $$\CC/[\AA] \xrightarrow{\sim} \Sub U$$
 where $U$ is an (explicitly constructed) injective object of $\BB$.
\end{theoremi}

Notice that we need and we prove more general versions of Theorems \ref{tho} and \ref{thc}, with more technical hypotheses and more precise conclusions.

The structure of this paper is as follows. In Section \ref{s:orders}, we explain main results about orders over an arbitrary complete discrete valuation ring $R$, and provide more general and more detailed versions of Theorem \ref{tho}. We also give a systematic way to construct pairs $(A, e)$ satisfying the hypotheses of Theorem \ref{tho} for a prescribed algebra $B$. The results of Section \ref{s:orders} are proven in Section \ref{ss:pforders}. In Section \ref{exactstr}, we recall the basics of exact categories and we give sufficient conditions for an ideal quotient category $\EE/[\FF]$ of an exact category $\EE$ by a subcategory $\FF$ of projective-injective objects to inherit the exact structure of $\EE$. In Section \ref{categ}, we give extended versions of Theorem \ref{thc}. Finally, in Section \ref{parflag}, we prove Theorem \ref{thparflag}.

\subsection*{Acknowledgement}
 We would like to thank Alastair King and Bernard Leclerc for valuable discussion about this topic.

\section{Main results} \label{s:orders}

\subsection{Orders} \label{ss:orders}
 
Let $R$ be a complete discrete valuation ring and $K$ be its field of fractions. Let $A$ be an \emph{$R$-order}, \emph{i.e.} an $R$-algebra which is free of finite rank as an $R$-module. We denote by $\fl A$ the full subcategory of $\mod A$ consisting of finite length $A$-modules, or equivalently $A$-modules which are finite length over $R$. Recall that, in this context, a finitely generated $A$-module $X$ is \emph{(maximal) Cohen-Macaulay} if the following equivalent conditions are satisfied:
 \begin{enumerate}[\rm (i)]
  \item $X$ is free (of finite rank) as an $R$-module;
  \item $\Hom_A(\fl A, X) = 0$, or equivalently $\soc X = 0$;
  \item $\Ext^1_A(X, \Hom_R(A, R)) = 0$, or equivalently, for any $i > 0$, $\Ext^i_A(X, \Hom_R(A, R)) = 0$.
 \end{enumerate}
We denote by $\CM A$ the exact full subcategory of $\mod A$ consisting of Cohen-Macaulay $A$-modules. Since $A$ is an $R$-order, both $A$ and $\Hom_R(A, R)$ are in $\CM A$. It is clear from (ii) that $(\fl A, \CM A)$ is a torsion pair in $\mod A$, which can be seen as coming from the co-tilting $A$-module $\Hom_R(A, R)$.

For an idempotent $e$ of $A$, we consider a full subcategory of $\CM A$:
\[ \CM_e A := \{X \in \CM A \mid e X \in \proj(e A e)\}.\]
This is clearly closed under extensions, and hence forms an exact category naturally. If $eAe$ is a hereditary order (\emph{i.e.} $\gl eAe = 1$), then $\CM_eA=\CM A$ holds because $\CM(eAe)=\proj(eAe)$.

Our first main Theorem, generalizing \cite{JeKiSu}, is the following one:
\begin{theorem}\label{mainA}
 Let $A$ be an $R$-order, and $e$ be an idempotent of $A$.
Assume that the following conditions are satisfied:
\begin{itemize}
\item $B:=A/(e)$ satisfies $\length_RB<\infty$.
\item There is an idempotent $g \in A$ such that $\add A e = \add \Hom_R(g A, R)$ as $A$-modules.
\end{itemize}
Then the following assertions hold.
\begin{enumerate}[\rm (a)]
  \item We have an equivalence of exact categories
\[F= B \otimes_A -:  (\CM_e A)/[A e] \xrightarrow{\sim} \Sub Q_g\]
 where $Q_g$ is the injective $B$-module associated with the image of the idempotent $g$ in $B$. 
 \item A quasi-inverse of $F$ is $\Hom_R(\Omega_A \Hom_R(-, K/R), R)$
  where $\Omega_A$ is the syzygy over $A$.
\end{enumerate}
We assume in addition that the following hypotheses hold:
\begin{itemize}
 \item There is an idempotent $f \in A$ such that $\add A f = \add \Hom_R(e A, R)$ as $A$-modules.
 \item $e A e$ is a Gorenstein order.
\end{itemize}
Then the following conclusions hold:
 \begin{enumerate}[\rm (a)] \rs{2}
  \item The module $Q_g$ is a projective $B$-module satisfying $\add Q_g  = \add Bf $.
  \item If $A \in \CM_e A$, then $\Sub Q_g = \Sub B$.
 \end{enumerate}
We suppose in addition that $A$ and $\Hom_R(A, R)$ are in $\CM_e A$.
 \begin{enumerate}[\rm (a)] \rs{4}
   \item  The order $A$ is Gorenstein if and only if $B$ is Iwanaga-Gorenstein of dimension at most one, \emph{i.e.} $\injdim {}_B B \leq 1$ and $\injdim B_B \leq 1$.
   \item If the conditions in (e) are satisfied, then we have triangle equivalences
\[\underline{\CM}_e A\cong\underline{\Sub} Q_g = \underline{\Sub} B.\]
 where $\underline{\CM}_e A := (\CM_e A)/[A]$ and $\underline{\Sub} B = (\Sub B)/[B]$.
 \end{enumerate}
\end{theorem}

Corollary \ref{thselfinj} presented in the introduction is an immediate consequence of Theorem \ref{mainA} as it is immediate that $Ae \cong \Hom_R(g A, R)$ for $g := 1-e$ in that case. In this paper, a more general version of Theorem \ref{mainA} plays an important role. Again let $A$ be an $R$-order and $e$ an idempotent of $A$. Let
\[\mod_e A := \{X \in \mod A \mid e X \in \proj(eAe) \}.\]
We consider the following conditions:
\begin{itemize}
 \item[(E1)\hphantom{\pup*}] $Ae$ is injective in $\CM_e A$, or equivalently, $\Ext^1_A(\CM_e A, Ae) = 0$;
 \item[(E2)\hphantom{\pup*}] $\Ext^2_{\mod_e A}(\mod_e A, Ae) = 0$;
 \item[(E2)\pup*] $\Ext^2_A(\mod_e A, Ae) = 0$.
\end{itemize}

We recall the definition of $\Ext^i_\EE$'s in Section \ref{exactstr} for exact categories $\EE$. Notice that, for a subcategory $\EE$ of $\mod A$, $\Ext^i_{\EE}$ is not necessarily the restriction of $\Ext^i_{A}$, except for $i = 1$.
In Lemma \ref{puprem}, we prove the following implications:
\begin{itemize}
 \item We have (E2)\pup* $\Rightarrow$ (E2).
 \item If $Ae=\Hom_R(gA,R)$ for some idempotent $g \in A$, then (E1) and (E2)\pup* are satisfied.
 \item If (E1) is satisfied and $A \in \CM_e A$, then (E2)\pup* is satisfied.
\end{itemize} 

Theorem \ref{mainA} follows from the following result:

\begin{theorem}\label{mainB}
Let $A$ be an $R$-order and $e$ an idempotent of $A$ such that $B:=A/(e)$ satisfies $\length_RB<\infty$.
\begin{enumerate}[\rm (a)]
 \item $(\add A e, \mod B)$ and $(\mod B, \CM_e A)$ are torsion pairs in $\mod_e A$.
 \item Let $\EE_1 := \{X\in\mod_e A\mid \Ext^1_A(X,Ae)=0\}$. We have an equivalence
\begin{equation}\label{functor F}
B \otimes_A -:  \EE_1/[Ae] \xrightarrow{\sim} \mod B.
\end{equation}
\end{enumerate}
\noindent If (E1) is satisfied, then the following assertion holds.
\begin{enumerate}[\rm (a)]  \rs{2}
 \item Let $U:=\Hom_A(B,Ae\otimes_R(K/R))\in\mod B$ where $K$ is the fraction field of $R$. The equivalence \eqref{functor F} restricts to an equivalence
\begin{equation}\label{functor F2}
B \otimes_A -:  (\CM_e A)/[A e] \xrightarrow{\sim} \Sub U.
\end{equation}
\end{enumerate}
If (E1) and (E2) are satisfied, then the following assertions hold.
 \begin{enumerate}[\rm (a)] \rs{3}
  \item $U$ is an injective $B$-module.
  \item \eqref{functor F} and \eqref{functor F2} are equivalences of exact categories, where $\EE_1/[Ae]$ and $(\CM_e A)/ [Ae]$ inherit canonically the exact structure of $\EE_1$ and $\CM_e A$ (see Section \ref{exactstr}).
  \item The exact categories $\EE_1$, $\CM_e A$, $\mod_e A$ and $\Sub U$ have enough projective objects and enough injective objects.
  \item Let $P$ be a projective cover of $\soc U$ as a $B$-module. Then, we have the equality $\EE_1 = \{X \in \mod_e A \mid \Hom_A(P, X) = 0\}$.
 \end{enumerate}
\end{theorem}

\subsection{Change of orders} \label{chgord}
We give a systematic method to construct pairs of orders and their idempotents which satisfy the conditions (E1) and (E2).

Let $A$ be an $R$-order, $e$ an idempotent of $A$ and $B$ a factor algebra of $A/(e)$. We suppose that the following two conditions are satisfied.
\begin{enumerate}[\rm (C1)]
\item $\length_RB<\infty$;
\item $B\in\Sub(Ae\otimes_R(K/R))$.
\end{enumerate}

Let $\mod_e^BA$ be the category of all $X\in\mod A$ such that there exists an exact sequence
\[0\rightarrow P\rightarrow X\rightarrow Y\rightarrow 0\]
with $P\in\add Ae$ and $Y\in\mod B$. Let $\CM_e^BA:=\CM A\cap\mod_e^BA$ and consider the condition:
\begin{enumerate}[\rm (C1)] \rs{2}
 \item $\Ext^1_{A}(\CM_{e}^{B}A,Ae)=0$.
\end{enumerate}

We will construct a new order $A'$ under this setting. Thanks to (C2), there is a monomorphism $\iota: B \hookrightarrow (Ae \otimes_R (K/R))^{\oplus \ell}$.  Applying $Ae^{\oplus \ell} \otimes_R -$ to the exact sequence $0 \rightarrow R \rightarrow K \rightarrow K/R \rightarrow 0$ and taking a pullback via $\iota$, we get a short exact sequence
\[0\rightarrow P\rightarrow \widetilde{B}\rightarrow B\rightarrow 0\]
with $P\in\add Ae$ and $\widetilde{B}\in\CM A$. We clearly have $\widetilde{B} \in \CM_e^B A$. Using (C3), one can check $\tilde B$ is independent of the choice of $\iota$ up to a direct summand in $\add Ae$ (see Theorem \ref{categ1} (a)).
\begin{itemize}
\item Let $W:=Ae\oplus\widetilde{B}$ and $A':=\End_A(W)$.
\end{itemize}
We can regard naturally $e$ as an idempotent of $A'$. Notice that $A'$ is uniquely defined up to Morita equivalence.

\begin{theorem} \label{thmchangidem2}
We assume that (C1), (C2) and (C3) hold. Then the following assertions hold.
\begin{enumerate}[\rm (a)]
\item We have a canonical isomorphism $B\cong A'/(e)$ of $R$-algebras.
\item We have (E1) $\Ext^1_{A'}(\CM_{e}A',A'e)=0$ and 
(E2)\pup* $\Ext^2_{A'}(\mod_{e}A',A'e)=0$.
\item Let $U:=\Hom_{A'}(B,A'e\otimes_R(K/R))\in\mod B$.
Then $U$ is an injective $B$-module and we have an equivalence of exact categories
\[B\otimes_{A'}-:(\CM_{e}A')/[A'e]\xrightarrow{\sim}\Sub U.\]
\item The functors $\Hom_A(W, -) : \mod A \to \mod A'$ and $W \otimes_{A'} - : \mod A' \to \mod A$ induce quasi-inverse equivalences of exact categories between $\mod_e^B A$ and $\mod_e A'$ on the one hand, and between $\CM_e^B A$ and $\CM_e A'$ on the other hand.
\item We have a commutative diagram
\[\xymatrix{
\CM_eA'\ar[rr]^{B\otimes_{A'}-}\ar[d]_{W\otimes_{A'}-}^\wr
&&\Sub U\ar@{^{(}->}[d]\\
\CM_e^BA\ar[rr]_{B\otimes_A-}&&\mod B
}\]
where all functors induce isomorphisms of $\Ext^1$ and the left side is an equivalence of exact categories.
\end{enumerate}
\end{theorem}

  Let us finally introduce a simple criterion for (C1), (C2) and (C3) to be satisfied:
 \begin{lemma} \label{lemmagorchangeorders}
  Let $A$ be an $R$-order, $e$ an idempotent of $A$ and $B$ a factor algebra of $A/(e)$. Let us assume that there exists an idempotent $g \in A$ such that $Ae \cong \Hom_R(g A, R)$. Then (C3) holds. Moreover, if (C1) holds, then (C2) holds if and only if $(1-g) \soc B = 0$.
 \end{lemma}

 We will prove Lemma \ref{lemmagorchangeorders} at the end of Subsection \ref{proofofchangidem}.

  \newcommand{\smar}[1]%
  {%
   \left[\begin{array}{r@{/}l} #1 \end{array}\right]%
  }
  \newcommand{\smarb}[1]%
  {%
   \scalebox{.7}{\renewcommand{\arraystretch}{.8} $\left[\begin{array}{r@{/}lcr@{/}l} #1 \end{array}\right]$}%
  }

\setlength{\fboxsep}{1pt}

  \newboxedcommand{\Sun}{\xymatrix@R=.2cm@C=.2cm@!=0cm{1}}
  \newboxedcommand{\Sdeux}{\xymatrix@R=.2cm@C=.2cm@!=0cm{2}}
  \newboxedcommand{\Strois}{\xymatrix@R=.2cm@C=.2cm@!=0cm{3}}
  \newboxedcommand{\Pun}{\xymatrix@R=.2cm@C=.2cm@!=0cm{1 & & \\ & 2 & \\ & & 3}}
  \newboxedcommand{\Pdeux}{\xymatrix@R=.2cm@C=.2cm@!=0cm{ & 2 & \\ 1 &  & 3 \\ & 2 & }}
  \newboxedcommand{\Ptrois}{\xymatrix@R=.2cm@C=.2cm@!=0cm{  &  & 3\\  & 2 &  \\ 1 & & }}
  \newboxedcommand{\dut}{\xymatrix@R=.2cm@C=.2cm@!=0cm{ & 2 & \\ 1 &  & 3 }}
  \newboxedcommand{\du}{\xymatrix@R=.2cm@C=.2cm@!=0cm{ & 2 \\ 1 &  }}
  \newboxedcommand{\dt}{\xymatrix@R=.2cm@C=.2cm@!=0cm{ 2 & \\  & 3 }}
  \newboxedcommand{\ud}{\xymatrix@R=.2cm@C=.2cm@!=0cm{ 1 & \\  & 2 }}
  \newboxedcommand{\td}{\xymatrix@R=.2cm@C=.2cm@!=0cm{ & 3 \\ 2 & }}
  \newboxedcommand{\utd}{\xymatrix@R=.2cm@C=.2cm@!=0cm{ 1 &  & 3 \\ & 2 & }}
 
  \newboxedcommand{\lSun}{\xymatrix@R=.3cm@C=.3cm@!=0cm{01 \\ 03 \\ 03 \\ 03 \\ 11 \\ 11}}
  \newboxedcommand{\lStrois}{\xymatrix@R=.3cm@C=.3cm@!=0cm{02 \\ 02 \\ 12 \\ 12 \\ 02 \\ 02}}
  \newboxedcommand{\lpSun}{\xymatrix@R=.3cm@C=.3cm@!=0cm{00 \\ 02 \\ 02 \\ 02 \\ 10 \\ 10}}
  \newboxedcommand{\lpStrois}{\xymatrix@R=.3cm@C=.3cm@!=0cm{01 \\ 01 \\ 11 \\ 11 \\ 01 \\ 01}}
  \newboxedcommand{\lPun}{\xymatrix@R=.3cm@C=.3cm@!=0cm{02 \\ 02 \\ 02 \\ 02 \\ 02 \\ 02}}
  \newboxedcommand{\lPtrois}{\xymatrix@R=.3cm@C=.3cm@!=0cm{01 \\ 03 \\ 03 \\ 03 \\ 01 \\ 01}}
  \newboxedcommand{\ldut}{\xymatrix@R=.3cm@C=.5cm@!=0cm@M=0cm{02 & 01 \\ 02 & 03 \\ 12 & 03 \\ 02 \ar@{-}[dr] & 03 \\ 02 & 01 \\ 02 & 11}}
  \newboxedcommand{\ldu}{\xymatrix@R=.3cm@C=.3cm@!=0cm{01 \\ 03 \\ 03 \\ 03 \\ 01 \\ 11}}
  \newboxedcommand{\ldt}{\xymatrix@R=.3cm@C=.3cm@!=0cm{02 \\ 02 \\ 12 \\ 02 \\ 02 \\ 02}}
  \newboxedcommand{\lPunp}{\xymatrix@R=.3cm@C=.3cm@!=0cm{00 \\ 02 \\ 12 \\ 02 \\ 10 \\ 10}}
  \newboxedcommand{\lPtroisp}{\xymatrix@R=.3cm@C=.3cm@!=0cm{01 \\ 01 \\ 11 \\ 11 \\ 01 \\ 11}}

  \newboxedcommand{\kSun}{\xymatrix@R=.3cm@C=.3cm@!=0cm{02 \\ 02 \\ 02 \\ 02 \\ 12 \\ 12}}
  \newboxedcommand{\kSdeux}{\xymatrix@R=.3cm@C=.3cm@!=0cm{01 \\ \fbox{01} \\ 02 \\ 11 \\ \fbox{11}+02 \\ 12}}
  \newboxedcommand{\kStrois}{\xymatrix@R=.3cm@C=.3cm@!=0cm{01 \\ 02 \\ 03 \\ 11 \\ 12 \\ 03}}
  \newboxedcommand{\kpSun}{\xymatrix@R=.3cm@C=.3cm@!=0cm{01 \\ 01 \\ 01 \\ 01 \\ 11 \\ 11}}
  \newboxedcommand{\kpStrois}{\xymatrix@R=.3cm@C=.3cm@!=0cm{00 \\ 01 \\ 02 \\ 10 \\ 11 \\ 02}}
  \newboxedcommand{\kPun}{\xymatrix@R=.3cm@C=.3cm@!=0cm{01 \\ 02 \\ 03 \\ 01 \\ 02 \\ 03}}
  \newboxedcommand{\kPdeux}{\xymatrix@R=.3cm@C=.3cm@!=0cm{02 \\ \fbox{02} \\ 03 \\ 02 \\ \fbox{02} \\ 03}}
  \newboxedcommand{\kpPdeux}{\xymatrix@R=.3cm@C=.3cm@!=0cm{01 \\ \fbox{01} \\ 02 \\ 01 \\ \fbox{01} \\ 02}}
  \newboxedcommand{\kPtrois}{\xymatrix@R=.3cm@C=.3cm@!=0cm{02 \\ 02 \\ 02 \\ 02 \\ 02 \\ 02}}
  \newboxedcommand{\kdut}{\xymatrix@R=.3cm@C=.5cm@!=0cm@M=0cm{02 & 01 \\ 02 & 02 \\ 02 & 03 \\ 02 & 11 \\ 02 \ar@{-}[r] & 02 \\ 12 & 03}}
  \newboxedcommand{\kdu}{\xymatrix@R=.3cm@C=.3cm@!=0cm{02 \\ 02 \\ 02 \\ 02 \\ 02 \\ 12}}
  \newboxedcommand{\kdt}{\xymatrix@R=.3cm@C=.3cm@!=0cm{01 \\ 02 \\ 03 \\ 11 \\ 02 \\ 03}}
  \newboxedcommand{\kud}{\xymatrix@R=.3cm@C=.3cm@!=0cm{01 \\ \fbox{01} \\ 02 \\ 01 \\ \fbox{11} + 02 \\ 12}}
  \newboxedcommand{\ktd}{\xymatrix@R=.3cm@C=.3cm@!=0cm{01 \\ \fbox{01} \\ 02 \\ 11 \\ \fbox{11} + 02 \\ 02}}
  \newboxedcommand{\kutd}{\xymatrix@R=.3cm@C=.3cm@!=0cm{01 \\ \fbox{01} \\ 02 \\ 01 \\ \fbox{11} + 02 \\ 02}}

  \newboxedcommand{\kPunp}{\xymatrix@R=.3cm@C=.3cm@!=0cm{01 \\ 01 \\ 01 \\ 11 \\ 11 \\ 11}}
  \newboxedcommand{\kPdeuxp}{\xymatrix@R=.3cm@C=.3cm@!=0cm{01 \\ \fbox{01} \\ 02 \\ 11 \\ \fbox{11} \\ 12}}
  \newboxedcommand{\kPtroisp}{\xymatrix@R=.3cm@C=.3cm@!=0cm{00 \\ 01 \\ 02 \\ 10 \\ 11 \\ 12}}

In the rest of this subsection we give an example illustrating Theorem \ref{thmchangidem2}. Let $B = \Pi$ be the preprojective algebra of type $A_3$ over a field $k$. In other terms
 $$\Pi = k \left.\left(\xymatrix{1 \ar@/^/[r]^{\alpha_1} & 2 \ar@/^/[r]^{\alpha_2} \ar@/^/[l]^{\beta_1} & 3 \ar@/^/[l]^{\beta_2}}\right) \right/ (\alpha_1 \beta_1, \alpha_2 \beta_2 - \beta_1 \alpha_1, \beta_2 \alpha_2).$$
 This algebra can also be realized as the following subquotient of the matrix algebra $M_3(k[\epsilon])$:
$$\Pi = \left[\begin{array}{r@{/}lr@{/}lr@{/}l}
       k[\varepsilon]&(\varepsilon) & k[\varepsilon]&(\varepsilon) & k[\varepsilon]&(\varepsilon) \\
       (\varepsilon)&(\varepsilon^2) & k[\varepsilon]&(\varepsilon^2) & k[\varepsilon]&(\varepsilon) \\
       (\varepsilon^2)&(\varepsilon^3) & (\varepsilon)&(\varepsilon^2) & k[\varepsilon]&(\varepsilon)
      \end{array}\right].$$
Let us denote $R := k\llbracket t \rrbracket$ and $S := R [\varepsilon] $. The $R$-order considered in Corollary \ref{thselfinj} is
 $$A := \left[\begin{array}{r@{/}lr@{/}lr@{/}lr@{/}lr@{/}lr@{/}l}
         S&(\varepsilon) & S&(\varepsilon) & S&(\varepsilon) & S&(\varepsilon) & S&(\varepsilon) & S&(\varepsilon) \\
         (\varepsilon)&(\varepsilon^2) & S&(\varepsilon^2) & S&(\varepsilon) & (\varepsilon)&(\varepsilon^2) & S&(\varepsilon^2) & S&(\varepsilon) \\  
         (\varepsilon^2)&(\varepsilon^3) & (\varepsilon)&(\varepsilon^2) & S&(\varepsilon) & (\varepsilon^2)&(\varepsilon^3) & (\varepsilon)&(\varepsilon^2) & S&(\varepsilon) \\  
         (t)&(t\varepsilon) & (t)&(t\varepsilon) & (t)&(t\varepsilon) & S&(\varepsilon) & S&(\varepsilon) & S&(\varepsilon) \\
         (t \varepsilon)&(t\varepsilon^2) & (t)&(t\varepsilon^2) & (t)&(t\varepsilon) & (\varepsilon)&(\varepsilon^2) & S&(\varepsilon^2) & S&(\varepsilon) \\  
         (t \varepsilon^2)&(t\varepsilon^3) & (t \varepsilon)&(t\varepsilon^2) & (t)&(t\varepsilon) & (\varepsilon^2)&(\varepsilon^3) & (\varepsilon)&(\varepsilon^2) & S&(\varepsilon)  
        \end{array}\right].$$
\begin{figure}
 $$\xymatrix@R=2cm@C=2cm@!=0cm{
   \ar@{--}[d] & & & \kPtrois \ar[dr] \ar[rr]^{t\varepsilon^{-1}} & & \kPunp \ar[dr] & \ar@{--}[d] \\
   \kStrois \ar[dr] \ar@{--}[d]^{}="m1" & & \kdu \ar[ur] \ar[dr] \ar@{..>}[ll] & & \ktd \ar[dr] \ar[ur] \ar@{..>}[ll] & & \kpSun \ar@{..>}[ll] \ar@{--}[d] \\
   \kPdeux \ar[r] \ar@{--}[d] \ar@/_1cm/[rr]_{\varepsilon^{-1}} & \kdut \ar[ur] \ar[dr] \ar@{..} ;"m1" \ar[r] & \kPdeuxp \ar[r] & \kSdeux \ar[ur] \ar[dr] \ar@/_1cm/@{..>}[ll] & & \kutd \ar[ur] \ar[dr] \ar[r] \ar@{..>}[ll]   & \kpPdeux \ar@{--}[d]^{}="m2" \ar@{..>} "m2";[l] \\
   \kSun \ar[ur] \ar@{--}[d] & & \kdt \ar[ur] \ar[dr] \ar@{..>}[ll] & & \kud \ar[ur] \ar@{..>}[ll] \ar[dr] & & \kpStrois \ar@{..>}[ll] \ar@{--}[d] \\
   & & & \kPun \ar[ur] \ar[rr]_{t\varepsilon^{-1}} & & \kPtroisp \ar[ur] & 
  }$$
  \caption{Auslander-Reiten quiver of $\CM_e A$.} \label{AR1a}
 \end{figure}
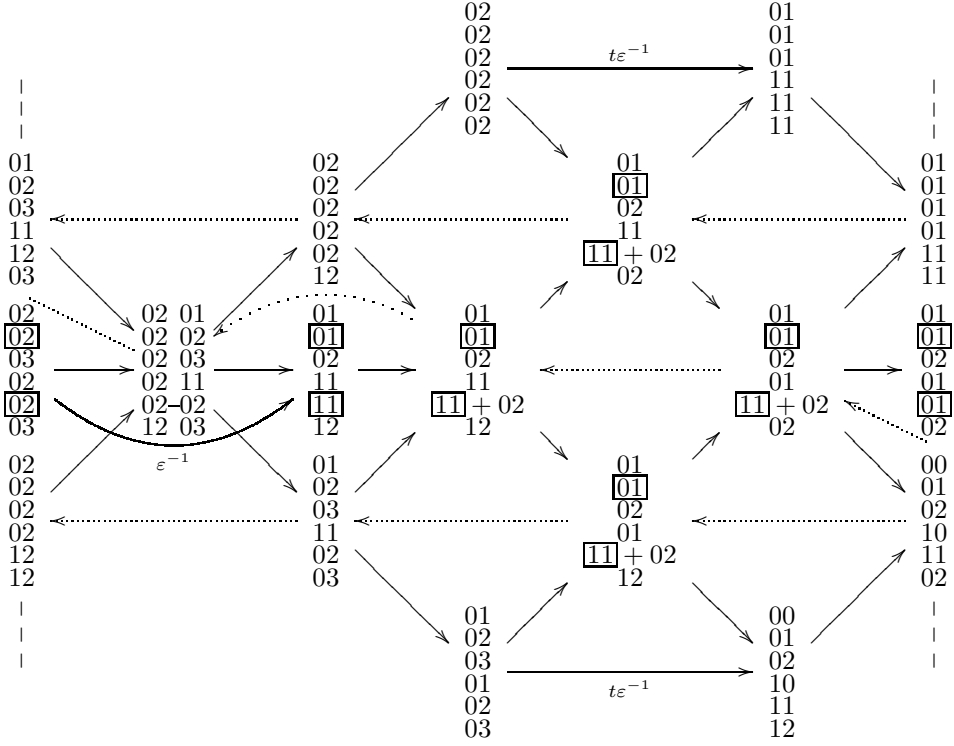

On Figure \ref{AR1a}, we draw the Auslander-Reiten quiver of $\CM_e A$, with notations $ij := (t^i \varepsilon^j)/(t^i \varepsilon^{j+1})$, $\fbox{$ij$} := (t^i \varepsilon^j)/(t^i \varepsilon^{j+2})$, and $ij \text{---} ij := \{(\varepsilon^j p, \varepsilon^j q) \in ij \times ij \mid p - q \in (t,\varepsilon)/(\varepsilon) \}$. Thus, the identity of $S$ induces a map $ij \to i'j'$ if and only if $(j, i) \geq (j', i')$ for the lexicographic order and analogous rules can be computed for $\fbox{$ij$}$. All arrows are induced by multiplications by an element of $S$, which is $\pm 1$ when it is not specified.

 Let $e_3$, $e_2$, $e_1$, $g_1$, $g_2$ and $g_3$ be the idempotents corresponding, in this order, to the rows of the matrix. They satisfy $A e_i \cong \Hom_R(g_i A, R)$ and $A g_i \cong \Hom_R(e_i A, R)$ as $A$-modules. We fix the idempotent $e = e_1 + e_2 + e_3$. According to Corollary \ref{thselfinj}, we have an equivalence of exact categories $$(\CM_e A)/[Ae] \cong \mod \Pi.$$ On Figure \ref{AR1}, we draw the Auslander-Reiten quiver of $\CM_e A$, replacing 
 objects which are not in $\add Ae$ by their image by $F$ in $\Sub U = \mod \Pi$ (here $U = \Pi$). We obtain the Auslander-Reiten quiver of $\mod \Pi$ by removing framed objects. The general relation between Auslander-Reiten quivers of $\CM_e A$ and $\Sub U$ will be discussed in \cite{DeIy-2}. 
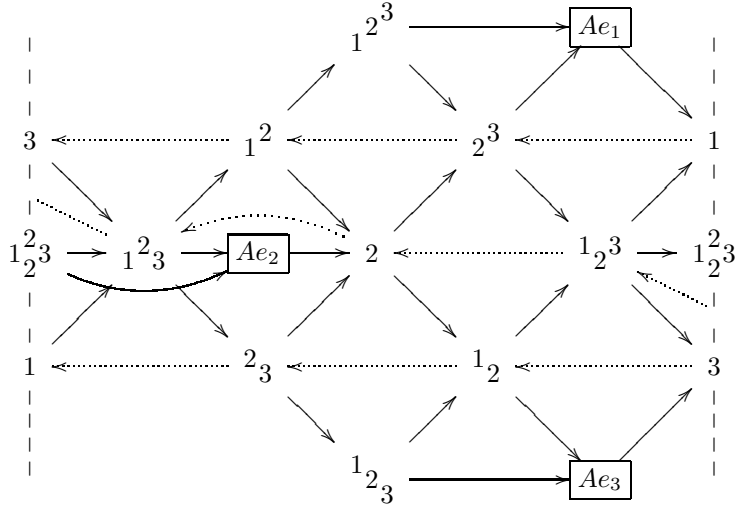
\begin{figure}
   $$\xymatrix@R=1.5cm@C=1.5cm@!=0cm{
   \ar@{--}[d] & & & \Ptrois \ar[dr] \ar[rr] & & *+[F]{Ae_1} \ar[dr] & \ar@{--}[d] \\
   \Strois \ar[dr] \ar@{--}[d]^{}="m1" & & \du \ar[ur] \ar[dr] \ar@{..>}[ll] & & \td \ar[dr] \ar[ur] \ar@{..>}[ll] & & \Sun \ar@{..>}[ll] \ar@{--}[d] \\
   \Pdeux \ar[r] \ar@{--}[d] \ar@/_.5cm/[rr] & \dut \ar[ur] \ar[dr] \ar@{..} ;"m1" \ar[r] & *+[F]{Ae_2} \ar[r] & \Sdeux \ar[ur] \ar[dr] \ar@/_.5cm/@{..>}[ll] & & \utd \ar[ur] \ar[dr] \ar[r] \ar@{..>}[ll]   & \Pdeux \ar@{--}[d]^{}="m2" \ar@{..>} "m2";[l] \\
   \Sun \ar[ur] \ar@{--}[d] & & \dt \ar[ur] \ar[dr] \ar@{..>}[ll] & & \ud \ar[ur] \ar@{..>}[ll] \ar[dr] & & \Strois \ar@{..>}[ll] \ar@{--}[d] \\
   & & & \Pun \ar[ur] \ar[rr] & & *+[F]{Ae_3} \ar[ur] & 
  }$$
  \caption{Auslander-Reiten quiver of $\CM_e A$. Objects are represented by their image by $F$ except objects of $\add Ae$.} \label{AR1}
 \end{figure}

 We explain the way to compute the minimal preimage of an object of $\Sub U$ by $F$ in this example. First, we know that preimages of simple modules are coradicals of indecomposable direct summands of $Ae$. Thus, we find $F(S^\circ_1) \cong S_1$, $F(S^\circ_2) \cong S_2$ and $F(S^\circ_3) \cong S_3$ where
 $$S^\circ_1 = \smar{S&(\varepsilon) \\ S&(\varepsilon) \\ S&(\varepsilon) \\   S&(\varepsilon) \\ (t)&(t\varepsilon) \\ (t)&(t\varepsilon)} \quad  \quad S^\circ_2 = \smar{S&(\varepsilon) \\ S&(\varepsilon^2) \\ (\varepsilon)&(\varepsilon^2) \\   (t)&(t\varepsilon) \\ (t, \varepsilon)&(\varepsilon^2) \\ (t\varepsilon)&(t\varepsilon^2)} \quad  \quad S^\circ_3 = \smar{S&(\varepsilon) \\ (\varepsilon)&(\varepsilon^2) \\ (\varepsilon^2)&(\varepsilon^3) \\   (t)&(t\varepsilon) \\ (t\varepsilon)&(t\varepsilon^2) \\ (\varepsilon^2)&(\varepsilon^3)}.$$ 
 Let us calculate the preimage $X^\circ$ of $\dut$ by $F$. There exists a pullback diagram
 $$\xymatrix{
  0 \ar[r] & A e_1 \oplus A e_3 \ar[r] \ar@{=}[d] & S^\circ_1 \oplus S^\circ_3 \ar@{^{(}->}[d] \ar[r] & S_1 \oplus S_3 \ar@{^{(}->}[d] \ar[r] & 0 \\
  0 \ar[r] & A e_1 \oplus A e_3 \ar[r] & X^\circ \ar@{->>}[d]  \ar[r] & \dut \ar[r] \ar@{->>}[d] & 0 \\
  & & S_2 \ar@{=}[r] & S_2
 } $$
 which permits us to get
 \newlength{\numerat} \settowidth{\numerat}{$(t\varepsilon^2)$}
 \newlength{\denomin} \settowidth{\denomin}{$(t\varepsilon^3)$}
 \newcommand{\cnstquot}[2]{\makebox[\numerat][r]{$#1$}/\makebox[\denomin][l]{$#2$}}
 $$X^\circ = \left[\vphantom{\begin{bmatrix}1\\1\\1\\1\\1\\1\\1\end{bmatrix}}\boxinminipage{\xymatrix@R=.5cm@C=1.7cm@!=0cm{
        \cnstquot{S}{(\varepsilon)} & \cnstquot{S}{(\varepsilon)} \\
        \cnstquot{S}{(\varepsilon)} & \cnstquot{(\varepsilon)}{(\varepsilon^2)} \\
        \cnstquot{S}{(\varepsilon)} & \cnstquot{(\varepsilon^2)}{(\varepsilon^3)} \\
        \cnstquot{S}{(\varepsilon)} & \cnstquot{(t)}{(t\varepsilon)} \\
        \cnstquot{S}{(\varepsilon)} \ar@{-} +<.6cm,0cm>;[r]+<-.6cm,0cm>  & \cnstquot{(\varepsilon)}{(\varepsilon^2)} \\
        \cnstquot{(t)}{(t\varepsilon)} & \cnstquot{(\varepsilon^2)}{(\varepsilon^3)} \\
        }}\right]$$
 where $[S/ (\varepsilon) \text{ --- } (\varepsilon)/ (\varepsilon^2)] := \{( x, \varepsilon y) \in S/(\varepsilon) \times (\varepsilon)/(\varepsilon^2) \mid x-y \in (t, \varepsilon)/(\varepsilon)\}$.

\begin{figure}
  $$\xymatrix@R=1.4cm@C=1.4cm@!=0cm{
   \ar@{--}[d] & & & \Ptrois \ar[dr] \ar@(dl,ul)[ddd] & \ar@{--}[d] \\
   \Strois \ar[dr] \ar@{--}[dd]^{} & & \du \ar[ur] \ar[r] \ar@{..>}[ll] & *+[F]{A'e'_3} \ar[r] & \Strois \ar@/^.5cm/@{..>}[ll] \ar@{--}[dd] \\
     & \dut \ar[ur] \ar[dr] & &  & \\
   \Sun \ar[ur] \ar@{--}[d] & & \dt \ar[dr] \ar@{..>}[ll] \ar[r] & *+[F]{A'e'_1} \ar[r] &  \Sun \ar@/_.5cm/@{..>}[ll] \ar@{--}[d] \\
   & & & \Pun \ar[ur] \ar@(ur,dr)[uuu] & 
  }
  \quad
\xymatrix@R=0.8cm@C=1.4cm@!=0cm{
   \ar@{--}[dd] & & & \lPtrois \ar[ddrr] \ar[dddd]_(.8){t \epsilon^{-1}} & & \ar@{--}[dd] \\
   \\
   \lStrois \ar[ddr] \ar@{--}[dddd]^{} & & \ldu \ar[uur] \ar[drr] \ar@{..>}[ll] & &  & \lpStrois \ar@/_/@{..>}[lll] \ar@{--}[dddd] \\
   & & &  & \lPtroisp \ar[ur] & \\
     & \ldut \ar[uur] \ar[ddr] & &   &  & \\
   & & &  \lPunp \ar[drr] &  & \\
   \lSun \ar[uur] \ar@{--}[dd] & & \ldt \ar[ddrr] \ar@{..>}[ll] \ar[ur] & & &  \lpSun \ar@/^/@{..>}[lll] \ar@{--}[dd] \\
   \\
   & & & & \lPun \ar[uur] \ar[uuuu]_(.8){ t \epsilon^{-1}} & 
  }
$$
  \caption{Auslander-Reiten quiver of $\CM_{e'} A'$. On the left diagram, objects are represented by their image by $F$ except objects of $\add A'e'$.} \label{AR2}
 \end{figure}
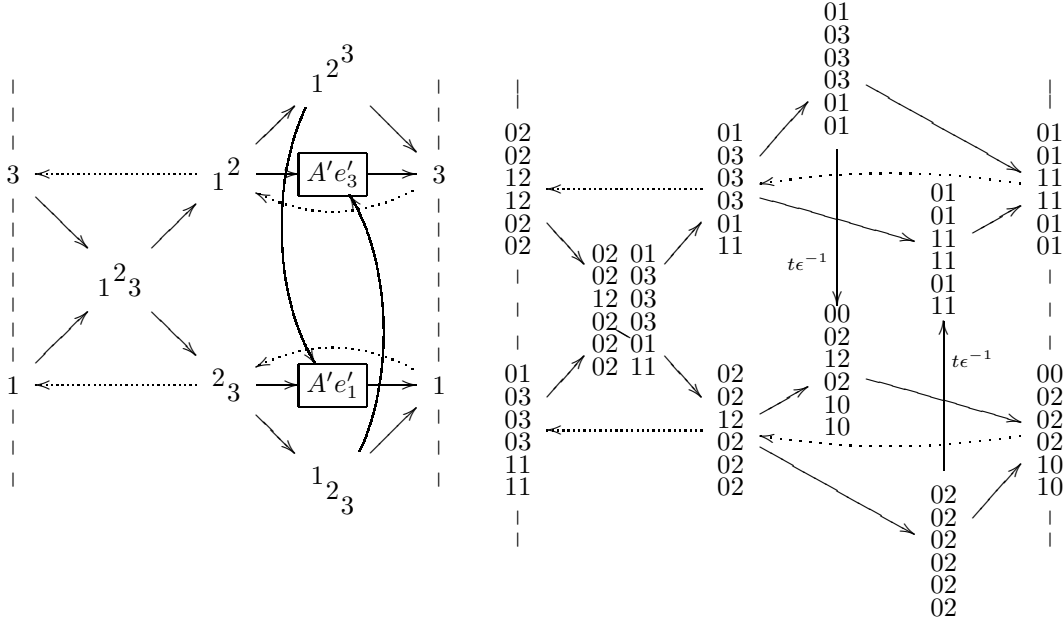

 Now, we apply Theorem \ref{thmchangidem2}. Let $e' := e_1 + e_3$ and $B' := \Pi/(\beta_1 \alpha_1)$. As a $B$-module,
  $$B' \cong \Pun \oplus \dut \oplus \Ptrois.$$
 Thanks to Lemma \ref{lemmagorchangeorders}, $B'$ and $e'$ satisfy the hypotheses of Theorem \ref{thmchangidem2}. Then, keeping notations of this subsection, we have $W = A e_1 \oplus A e_3 \oplus A g_1 \oplus X^\circ \oplus A g_3$. Then, $A' := \End_A (W)$ is easy to compute:

  $$A' = \left[\vphantom{\begin{bmatrix}1\\1\\1\\1\\1\\1\\1\end{bmatrix}}\boxinminipage{\xymatrix@R=.5cm@C=1.7cm@!=0cm{
        \cnstquot{S}{(\varepsilon)} & \cnstquot{S}{(\varepsilon)} & \cnstquot{S}{(\varepsilon)} & \cnstquot{S}{(\varepsilon)} & \cnstquot{S}{(\varepsilon)} & \cnstquot{S}{(\varepsilon)} \\
         \cnstquot{(\varepsilon^2)}{(\varepsilon^3)} & \cnstquot{S}{(\varepsilon)} & \cnstquot{S}{(\varepsilon)} & \cnstquot{S}{(\varepsilon)} & \cnstquot{(\varepsilon^2)}{(\varepsilon^3)} & \cnstquot{(\varepsilon^2)}{(\varepsilon^3)} \\  
         \cnstquot{(t\varepsilon^2)}{(t\varepsilon^3)} & \cnstquot{(t)}{(t\varepsilon)} & \cnstquot{S}{(\varepsilon)} & \cnstquot{(t)}{(t\varepsilon)} & \cnstquot{(\varepsilon^2)}{(\varepsilon^3)} & \cnstquot{(\varepsilon^2)}{(\varepsilon^3)} \\  
         \cnstquot{(\varepsilon^2)}{(\varepsilon^3)} & \cnstquot{(t)}{(t\varepsilon)} & \cnstquot{S}{(\varepsilon)} & \cnstquot{S}{(\varepsilon)} \ar@{-} +<.6cm,0cm>;[dr]+<-.4cm,.1cm> & \cnstquot{(\varepsilon^2)}{(\varepsilon^3)} & \cnstquot{(\varepsilon^2)}{(\varepsilon^3)} \\
         \cnstquot{(t)}{(t\varepsilon)} & \cnstquot{S}{(\varepsilon)} & \cnstquot{S}{(\varepsilon)} & \cnstquot{S}{(\varepsilon)} & \cnstquot{S}{(\varepsilon)} & \cnstquot{S}{(\varepsilon)} \\  
         \cnstquot{(t)}{(t\varepsilon)} & \cnstquot{(t)}{(t\varepsilon)} & \cnstquot{S}{(\varepsilon)} & \cnstquot{S}{(\varepsilon)} & \cnstquot{(t)}{(t\varepsilon)} & \cnstquot{S}{(\varepsilon)}
        }}\right]$$
  where $[S/(\varepsilon) \text{ --- } S/(\varepsilon)] := \{(x, y) \in S/(\varepsilon) \times S/(\varepsilon) \mid x-t \in (t)/(t\varepsilon)\}$.

  Thanks to Theorem \ref{thmchangidem2}, we have $(\CM_{e'} A')/[Ae']$ is equivalent to the subcategory of $\mod \Pi$ consisting of modules whose socle is supported at vertices $1$ and $3$. To illustrate this fact, we give two representations of the Auslander-Reiten quiver of $\CM_{e'} A'$ on Figure \ref{AR2}.

\subsection{Notations}
In this paper, if $f: X \rightarrow Y$ and $g: Y \rightarrow Z$ are two morphisms in a category, we write $fg: X \rightarrow Z$ for the composed morphism. 

Let $\Ab$ be the category of abelian groups.
For an additive category $\AA$, an \emph{$\AA$-module} is a contravariant additive functor $F:\AA \rightarrow \Ab$.
We say that $\AA$-module $F$ is \emph{finitely generated} if there exists an epimorphism of functor $\Hom_\AA(\AA,X) \rightarrow F$ for some $X\in \AA$.

\section{Results on exact categories} \label{exactstr}

The aim of this section is to study ideal quotient categories $\EE/[\FF]$ of an exact category $\EE$ by a full subcategory $\FF$ consisting of projective-injective objects. More precisely, we study conditions for $\EE/[\FF]$ to inherit the exact structure of $\EE$. In particular, we prove that it is the case if and only if admissible monomorphisms and epimorphisms are mapped to categorical monomorphisms and epimorphisms by the canonical projection $\EE \rightarrow \EE/[\FF]$. This is a particular case of Theorem \ref{eq1}.

\subsection{Preliminaries about exact categories}

We recall here main definitions and elementary results about exact categories. We consider an additive category $\EE$ endowed with a family $\SS$ of pairs of morphisms $(f, g)$ of $\EE$ where $f$ is a kernel of $g$ and $g$ is a cokernel of $f$. We denote such a pair 
 $$0 \rightarrow X \xrightarrow{f} Y \xrightarrow{g} Z \rightarrow 0$$
and for $(f,g) \in \SS$, we call $(f,g)$ an \emph{admissible short exact sequence}, $f$ an \emph{admissible monomorphism} and $g$ an \emph{admissible epimorphism}. We call $(\EE, \SS)$ an \emph{exact category} if it satisfies the following axioms due to Quillen \cite{Qu73} modified by Keller \cite[Appendix A]{Ke90}:
\begin{itemize}
 \item[(Ex0)] $\SS$ is stable under isomorphisms and contains \emph{split short exact sequences} of the form
   $$0 \rightarrow X \xrightarrow{\begin{sbmatrix} \id_X & 0 \end{sbmatrix}} X \oplus Z \xrightarrow{\begin{sbmatrix} 0 \\ \id_Z \end{sbmatrix}} Z \rightarrow 0;$$
 \item[(Ex1)] The composition of two admissible epimorphisms is an admissible epimorphism;
 \item[(Ex1)${}^{\operatorname {op}}$] The composition of two admissible monomorphisms is an admissible monomorphism;
 \item[(Ex2)] For any admissible short exact sequence $$0 \rightarrow X \xrightarrow{f} Y \xrightarrow{g} Z \rightarrow 0$$ and morphism $v: Z' \rightarrow Z$, we can form a \emph{pullback diagram}, \emph{i.e.} a commutative diagram of the form:
 $$\xymatrix{ 
  0 \ar[r] & X \ar[r]^{f'} \ar@{=}[d] & Y' \ar[d]^{v'} \ar[r]^{g'} & Z' \ar[d]^v \ar[r] & 0 \\
  0 \ar[r] & X \ar[r]_f & Y \ar[r]_g & Z \ar[r] & 0
 }$$
 where the first row is an admissible short exact sequence;
 \item[(Ex2)${}^{\operatorname {op}}$] For any admissible short exact sequence $$0 \rightarrow X \xrightarrow{f} Y \xrightarrow{g} Z \rightarrow 0$$ and morphism $u: X \rightarrow X'$, we can form a \emph{pushout diagram}, \emph{i.e.} a commutative diagram of the form:
 $$\xymatrix{
  0 \ar[r] & X \ar[r]^f \ar[d]_u & Y \ar[d]_{u'} \ar[r]^g & Z' \ar@{=}[d] \ar[r] & 0 \\
  0 \ar[r] & X' \ar[r]_{f'} & Y' \ar[r]_{g'} & Z \ar[r] & 0. 
 }$$
where the second row is an admissible short exact sequence.
\end{itemize}

We often write $\EE$ instead of $(\EE, \SS)$ when we consider only one exact structure on $\EE$. When not specified, we use the term \emph{short exact sequence} (respectively, \emph{monomorphism}, \emph{epimorphism}) for \emph{admissible short exact sequence} (respectively, admissible monomorphism, admissible epimorphism). In contrast, we use \emph{categorical monomorphism} (respectively, \emph{categorical epimorphism}) for a monomorphism (respectively, epimorphism) which is not necessarily admissible.

We will use freely the following easy facts about exact categories:
   \begin{itemize}
    \item In (Ex2), we have the following admissible short exact sequence:
     $$0 \rightarrow Y' \xrightarrow{\begin{sbmatrix} v' & g' \end{sbmatrix}} Y \oplus Z' \xrightarrow{\begin{sbmatrix} g \\ -v \end{sbmatrix}} Z \rightarrow 0;$$
    \item In (Ex2), if $v$ is an admissible epimorphism, then $v'$ is also one and $\ker v = (\ker v') g'$;
    \item In (Ex2), if $v$ is an admissible monomorphism, then $v'$ is also one and $\coker v' = g(\coker v)$;
    \item In (Ex2)${}^{\operatorname {op}}$, we have the following admissible short exact sequence:
     $$0 \rightarrow X \xrightarrow{\begin{sbmatrix} u & f \end{sbmatrix}} X' \oplus Y \xrightarrow{\begin{sbmatrix} f' \\ -u' \end{sbmatrix}} Y' \rightarrow 0;$$
    \item In (Ex2)${}^{\operatorname {op}}$, if $u$ is an admissible epimorphism, then $u'$ is also one and $\ker u' = (\ker u) f$;
    \item In (Ex2)${}^{\operatorname {op}}$, if $u$ is an admissible monomorphism, then so is $u'$ and $\coker u = f'(\coker u')$;
    \item If a morphism is an admissible monomorphism and an admissible epimorphism, then it is an isomorphism;
    \item If, in a morphism of short exact sequences, the left and right components are admissible monomorphisms (respectively, epimorphisms), then the middle one is;
    \item In (Ex2) and (Ex2)${}^{\operatorname {op}}$, the diagrams are uniquely determined up to unique isomorphisms.
   \end{itemize}

Let us recall the following definition:
\begin{definition}
 A functor $F$ between exact categories $(\EE, \SS)$ and $(\EE', \SS')$ is \emph{exact} if $F(\SS) \subset \SS'$. An object $X \in \EE$ is \emph{projective} (respectively, \emph{injective}) if $\Hom_\EE(X, -)$ (respectively, $\Hom_\EE(-, X)$) is exact. We say that $\EE$ has \emph{enough injective objects} (respectively, \emph{enough projective objects}) if for any $X \in \EE$ there exists a short exact sequence $0 \to X \to I \to Y \to 0$ (respectively, $0 \to Y \to P \to X \to 0$) in $\SS$ such that $I$ is injective (respectively, $P$ is projective).
\end{definition}

Recall that these notions permit to define extension functors $\Ext^i_\EE$ which satisfy the expected properties, either from Yoneda's structure of long exact sequences, or using projective resolutions if $\EE$ has enough projective objects, or using injective resolutions if $\EE$ has enough injective objects, or more generally using the derived category of $\EE$. 

Throughout this paper, we will use the following definition:
\begin{definition}
Let $\EE$ and $\EE'$ be exact categories and $F: \EE \rightarrow \EE'$ an exact functor. We say that $F$ is \emph{exact bijective} if the induced morphism $\Ext^1_\EE(-, -) \rightarrow \Ext^1_{\EE'}(F-, F-)$ is an isomorphism. We say that $F$ is an \emph{equivalence of exact categories} if it is an exact bijective equivalence of categories (or, equivalently, an exact equivalence of categories with an exact quasi-inverse).
\end{definition}

A typical example of exact bijective functor arises when $\EE$ is a full exact subcategory of $\EE'$ (\emph{i.e.} a full subcategory which is closed under extensions).

\begin{remark} \label{remfrob}
 Assume $F: \EE \rightarrow \EE'$ is a dense and exact bijective functor.
 \begin{enumerate}[\rm (a)]
  \item For any $X \in \EE$, $X$ is projective (respectively, injective) if and only if $FX$ is projective (respectively, injective).
  \item $\EE$ has enough projective (respectively, injective) objects if and only if $\EE'$ has enough projective (respectively, injective) objects. 
  \item $\EE$ is Frobenius if and only if $\EE'$ is Frobenius.
 \end{enumerate}
\end{remark}

We give an elementary result about second extension groups:

\begin{proposition} \label{ext2inj}
 Let $F : \EE \to \EE'$ be an exact bijective functor. Then, it induces a canonical natural monomorphism $\Ext^2_\EE(-,-) \hookrightarrow \Ext^2_\EE(F-,F-)$.
\end{proposition}

\begin{proof}
 The existence of a map $\varphi: \Ext^2_\EE(-,-) \rightarrow \Ext^2_\EE(F-,F-)$ is immediate. We consider an admissible $4$-terms exact sequence
  $\xi: 0 \to X \to Y_1 \to Y_2 \to Z \to 0$
 which, by definition, comes from two short exact sequences $$\xi_1: 0 \to X \to Y_1 \to Y \to 0 \quad \text{and} \quad \xi_2: 0 \to Y \to Y_2 \xrightarrow{u} Z \to 0.$$ Suppose that $\xi \in \ker \varphi_{Z, X}$. Applying $\Hom_\EE(-, X)$ and $\Hom_{\EE'}(-, X)$ to $\xi_2$ gives a commutative diagram of exact sequences:
 $$\xymatrix@C=1.2cm{
  \Ext^1_\EE(Y_2, X) \ar@{=}[d] \ar[r] &\Ext^1_\EE(Y, X) \ar@{=}[d] \ar[r] & \Ext^2_\EE(Z, X) \ar[d]^{\varphi_{Z, X}} \ar[r]^-{\Ext^2_\EE(u, X)} & \Ext^2_\EE(Y_2, X) \ar[d] \\
  \Ext^1_{\EE'}(Y_2, X) \ar[r] & \Ext^1_{\EE'}(Y, X) \ar[r] & \Ext^2_{\EE'}(Z, X) \ar[r] & \Ext^2_{\EE'}(Y_2, X).
 }$$
 By Definition of Yoneda product, $\xi \in \ker \Ext^2_\EE(u, X)$, so an easy diagram chasing gives $\xi = 0$.
\end{proof}

Let us define important concepts:

\begin{definition}
 Let $\EE$ be a Krull-Schmidt additive category and $\EE' \subset \EE$ an additive subcategory. 
 \begin{enumerate}[\rm (a)]
  \item We say that $f: X \to Y$ in $\EE$ is \emph{left minimal} if for any $g \in \End_\EE(Y)$ such that $fg = f$, $g$ is invertible, or equivalently if for any idempotent $e \in \End_\EE(Y)$, $fe = f$ implies $e = \id_Y$.
  \item We say that $g: Y \to X$ in $\EE$ is \emph{right minimal} if for any $f \in \End_\EE(Y)$ such that $fg = g$, $f$ is invertible, or equivalently if for any idempotent $e \in \End_\EE(Y)$, $eg = g$ implies $e = \id_Y$.
  \item We say that $f: X \to X'$ in $\EE$ is a \emph{left $\EE'$-approximation} (of $X$) if $X' \in \EE'$ and any morphism from $X$ to any object of $\EE'$ factors through $f$.
  \item We say that $g: X' \to X$ in $\EE$ is a \emph{right $\EE'$-approximation} (of $X$) if $X' \in \EE'$ and any morphism from any object of $\EE'$ to $X$  factors through $g$.
 \end{enumerate}
\end{definition}

Notice that, in the situation of the previous definition, if an object $X \in \EE$ admits a left (respectively right) $\EE'$-approximation, then it admits a left (respectively right) minimal $\EE'$-approximation which is unique up to isomorphism.

\subsection{Exact ideal quotients of an exact category}

Let $(\EE, \SS)$ be an exact category and $\EE'$ a full subcategory of $\EE$ which is closed under extensions.
Then $(\EE',\SS')$ forms an exact category for the family $\SS'$ of all admissible exact sequences in $\SS$ whose terms belong to $\EE'$.

We denote by $\FF$ a subcategory of $\EE$ satisfying $\Ext^1_\EE(\FF, \EE') = \Ext^1_\EE(\EE', \FF) = 0$. Let $\SS'_\FF$ be the class of pairs of morphisms in $\EE'/[\FF]$ which are isomorphic to a pair in $\pi(\SS')$ where $\pi: \EE' \rightarrow \EE'/[\FF]$ is the canonical functor.

\begin{theorem} \label{eq1}
 The following are equivalent:
 \begin{enumerate}[\rm (i)]
  \item $(\EE'/[\FF], \SS'_\FF)$ is exact; \label{eq1a}
  \item For any admissible monomorphism (respectively, epimorphism) $f$ of $(\EE', \SS')$, $\pi(f)$ is a categorical monomorphism (respectively, epimorphism) in $\EE'/[\FF]$. \label{eq1b}
 \end{enumerate}
 In this case, $\pi: \EE' \rightarrow \EE'/[\FF]$ is automatically exact bijective. 
\end{theorem}

\begin{proof}
 $\eqref{eq1a} \Rightarrow \eqref{eq1b}$ is trivial. Let us prove the converse. Let us first check that any $(\bar f, \bar g) \in \SS'_\FF$ is a kernel-cokernel pair. By \eqref{eq1b}, $\bar f$ is a monomorphism and $\bar g$ is an epimorphism. By definition, we can lift $(\bar f, \bar g)$ to $(f,g) \in \SS'$. Suppose that $\bar f \bar h = 0$ for some morphism $\bar h$ of $\EE' / [\FF]$. By definition, it means that there is a commutative diagram in $\EE$ of the form
 $$\xymatrix{
  0 \ar[r] & X \ar[r]^{f} \ar[d]_{h'} & Y \ar[r]^g \ar[d]^h & Z \ar[r] & 0 \\
  & F \ar[r]_{f'} & Z' 
 }$$
 with $F \in \FF$. As $\Ext^1_\EE(Z, F) = 0$, there exists $u: Y \rightarrow F$ such that $h' = fu$. Thus, $h = uf' + gv$ for some $v: Z \rightarrow Z'$ and $\bar h = \bar g \bar v$ holds.  It proves that $\bar g$ is a cokernel of $\bar f$. Dually, we prove that $\bar f$ is a kernel of $\bar g$.

 Let us check axioms of exact categories one by one:
 \begin{itemize}
  \item[(Ex0)] Obvious.
  \item[(Ex1)] Suppose that $\bar g: X \rightarrow X'$ and $\bar g': X' \rightarrow X''$ are epimorphisms in $\SS'_\FF$. It is easy to check that we can lift them to admissible epimorphisms $g: X \oplus F_1 \rightarrow X' \oplus F_2$ and $g': X' \oplus F_3 \rightarrow X'' \oplus F_4$ of $\EE'$. Thus $\bar g \bar g'$ can be lifted to an admissible epimorphism $X \oplus F_1 \oplus F_3 \rightarrow X'' \oplus F_2 \oplus F_4$ in $\EE'$ using (Ex1) in $(\EE', \SS')$. By definition, $\bar g \bar g'$ is then an epimorphism in $\SS'_\FF$.
  \item[(Ex1)${}^{\operatorname {op}}$] Dual of the previous.
  \item[(Ex2)] Let $\bar g: Y \rightarrow Z$ be an epimorphism in $\SS'_\FF$ and $\bar v: Z' \rightarrow Z$ be a morphism in $\EE'/[\FF]$. Without loss of generality, we can suppose that they come from lifts $g: Y \rightarrow Z$ and $v: Z' \rightarrow Z$ in $\EE'$ where $g$ is an admissible epimorphism. Thus, we can complete the pair to a pullback diagram
   $$\xymatrix{
    Y' \ar[r]^{g'} \ar[d]_{v'} & Z' \ar[d]^v\\
    Y \ar[r]_g & Z
   }$$
   in $\EE'$ where $g'$ is an admissible epimorphism. Then $0 \rightarrow Y' \rightarrow Z' \oplus Y \rightarrow Z \rightarrow 0$ is in $\SS'$, and its projection to $\EE'/[\FF]$ is in $\SS'_\FF$. Thus the diagram is also a pullback diagram in $\EE'/[\FF]$, and $\bar g'$ is an epimorphism in $\SS'_\FF$. 
  \item[(Ex2)${}^{\operatorname {op}}$] Dual of the previous. 
 \end{itemize} 

 We have finished proving the equivalence. Let us check that the projection $\pi: \EE' \rightarrow \EE'/[\FF]$ is exact bijective. First of all, for $X, Z \in \EE'$, the induced map $\Ext_{\EE'}^1(Z, X) \rightarrow \Ext_{\EE'/[\FF]}^1(\pi Z, \pi X)$ is clearly surjective. To prove that it is injective, let us consider a short exact sequence
 \begin{equation}\label{to split}
0 \rightarrow X \xrightarrow{f} Y \xrightarrow{g} Z \rightarrow 0
 \end{equation}
 which splits in $\EE'/[\FF]$. By definition, it means that there is $g': Z \rightarrow Y$ and two morphisms $u: Z \rightarrow F$ and $v: F \rightarrow Z$ with $F \in \FF$ such that $\id_Z = g' g + uv$. As $\Ext^1_\EE(F, X) = 0$, there exists $v': F \rightarrow Y$ such that $v = v'g$.
Thus $\id_Z = (g'+uv') g$ holds, and \eqref{to split} splits in $\EE'$. Therefore $\Ext_{\EE'}^1(Z, X) \rightarrow \Ext_{\EE/[\FF]}^1(\pi Z, \pi X)$ is injective. 
\end{proof}

In the rest of this section we give sufficient conditions for Theorem \ref{eq1} \eqref{eq1b} to hold. For two subcategories $\BB$ and $\CC$ of $\EE$, we denote by $\CC \searrow \BB$ the full subcategory of $\EE$ consisting of $X$ such that for any complex $Y \xrightarrow{g} B \xrightarrow{f} X$ with $B \in \BB$ and $Y \in \EE'$, there exists a morphism of complexes
$$\xymatrix{ 
 Y \ar[r]^g \ar[d] & B \ar[r]^f \ar[d] & X \ar@{=}[d] \\
 C \ar[r]_{g'} & B' \ar[r]_{f'} & X
}$$
with $B' \in \BB$ and $C \in \CC$. Notice that $[\EE' \searrow \BB] = \EE$ holds since we can choose $f'=f$ and $g'=g$.
Dually, we denote by $\BB \nearrow \CC$ the full subcategory of $\EE$ consisting of $X$ such that for any complex $X \xrightarrow{f} B \xrightarrow{g} Y$ with $B \in \BB$ and $Y \in \EE'$, there exists a morphism of complexes
$$\xymatrix{ 
 X \ar[r] \ar@{=}[d] & B' \ar[r] \ar[d] & C \ar[d] \\
 X \ar[r]_f & B \ar[r]_g & Y
}$$
with $B' \in \BB$ and $C \in \CC$. As before, we get $[\BB \nearrow \EE'] = \EE$. We get the following corollary:

\begin{corollary}
 Let $\PP$ (respectively, $\II$) be the full subcategory of $\EE$ consisting of objects $X$ satisfying $\Ext^1_\EE(X, \EE') = 0$ (respectively, $\Ext^1_\EE(\EE', X) = 0$). If $\EE' \subset (\FF \nearrow [\II \searrow \FF]) \cap ([\FF \nearrow \PP] \searrow \FF)$ then $(\EE'/[\FF], \SS'_\FF)$ is an exact category. 
\end{corollary}

\begin{proof}
We need to prove Theorem \ref{eq1} \eqref{eq1b}. We do it for admissible monomorphisms. Let $0 \rightarrow X \xrightarrow{f} Y \xrightarrow{g} Z \rightarrow 0$ be a short exact sequence in $\SS'$,
and let $u: X' \rightarrow X$ be a morphism such that $\bar u \bar f = 0$ in $\EE/[\FF]$.
Then $uf = f'u'$ holds for some $f': X' \rightarrow F'$ and $u': F' \rightarrow Y$ with $F' \in \FF$.

Suppose first that $X' \in [\FF \nearrow \PP]$. By definition, we can complete a commutative diagram 
 $$\xymatrix{
  & X' \ar[r]^\alpha \ar@{=}[d] & F \ar[r]^\beta \ar[d]^{v'} & P \ar[d]^{v''} \\
  & X' \ar[d]_u \ar[r]^{f'} & F' \ar[d]^{u'} \ar[r]^{u'g} & Z \ar@{=}[d] \\
  0 \ar[r] & X \ar[r]_f & Y \ar[r]_g & Z \ar[r] & 0
 }$$
with $\alpha \beta = 0$ and $F \in \FF$ and $P \in \PP$. As $\Ext^1_\EE(P, X) = 0$, $v'' = g''g$ for some $g'': P \rightarrow Y$ and we get easily $v'u' =  \beta g'' + f'' f$ for some $f'': F \rightarrow X$. 
We deduce that $\alpha f'' f = \alpha v'u' - \alpha \beta g'' = uf$. As $f$ is a monomorphism, $\alpha f'' = u$ and therefore $\bar u = 0$. 

Let us now suppose that $X' \in \EE'$. As $Z \in ([\FF \nearrow \PP] \searrow \FF)$, we can complete the following commutative diagram
$$\xymatrix{
  & A \ar[r]^\alpha & F \ar[r]^\beta & Z \\
  & X' \ar[d]_u \ar[r]^{f'} \ar[u]^v & F' \ar[d]^{u'} \ar[r]^{u'g} \ar[u]_{v'} & Z \ar@{=}[d] \ar@{=}[u]\\
  0 \ar[r] & X \ar[r]_f & Y \ar[r]_g & Z \ar[r] & 0
 }$$
with $\alpha \beta = 0$ and $F \in \FF$ and $A \in [\FF \nearrow \PP]$. Then, as $\Ext^1_\EE(F, X) = 0$, we get $\beta = \beta' g$ with $\beta': F \rightarrow Y$ and, as $f$ is the kernel of $g$, there exist $\alpha': A \rightarrow X$ such that $\alpha \beta' = \alpha' f$. As $A \in [\FF \nearrow \PP]$ and $\bar \alpha' \bar f =0$, by the first part of the argument, $\bar \alpha' = 0$.

Finally, by an easy diagram chasing, there exists $w: F' \rightarrow X$ such that $u' = v' \beta' + w f$. So we get $uf = f' u' = f' v' \beta' + f' w f = v \alpha \beta' + f' w f = v \alpha' f + f' w f$. As $f$ is a monomorphism, we deduce that $u = v \alpha' + f' w$. Thus $\bar u = 0$ holds since $\bar \alpha'=0$.
\end{proof}

In the rest of this section, we give three special cases as an application. Notice that the first case recovers Chen's result \cite[Theorem 3.1]{Ch12} for $\EE' = \EE$.

\begin{corollary} \label{cdexactm}
\begin{enumerate}[\rm (a)]
\item If, for any $X \in \EE'$, there exist left and right $\FF$-approximations and pseudo-cokernel and pseudo-kernel
  $$X \rightarrow F^X \rightarrow P^X \quad \text{and} \quad I_X \rightarrow F_X \rightarrow X$$
 such that $P^X \in \PP$ and $I_X \in \II$ then $(\EE'/[\FF], \SS'_\FF)$ is an exact category.
\item If, for any $X \in \EE' $ there exists a left $\FF$-approximation $X \rightarrow F^X$ which is a categorical epimorphism, then $(\EE'/[\FF], \SS'_\FF)$ is an exact category.
\item \label{cdexact} If, for any $X \in \EE'$ there exists a right $\FF$-approximation $F_X \rightarrow X$ which is a categorical monomorphism, then $(\EE'/[\FF], \SS'_\FF)$ is an exact category.
\end{enumerate}
\end{corollary}

\begin{proof} 
 \begin{enumerate}[\rm (a)]
  \item Let $X \rightarrow F \rightarrow Y$ be a complex where $X, Y \in \EE'$ and $F \in \FF$. It is easy to complete the following commutative diagram
  $$\xymatrix{
   X \ar@{=}[d] \ar[r] & F^X \ar[r] \ar[d] & P^X \ar[d] \\
   X \ar[r] & F \ar[r] & Y 
  }$$
  so $\EE' \subset [\FF \nearrow \PP]$. Thus we have $\EE=[\EE' \searrow \FF]\subset([\FF \nearrow \PP]\searrow\FF)$. Dually we have $\EE=(\FF \nearrow [\II\searrow\FF])$.
  \item By the same argument as the beginning of (a), we get $\EE' \subset [\FF \nearrow 0]$. So 
  $$\EE' \subset [\FF \nearrow 0] \subset (\FF \nearrow [\II \searrow \FF]) \quad \text{and} \quad \EE = [\EE' \searrow \FF] \subset ([\FF \nearrow 0] \searrow \FF) \subset ([\FF \nearrow \PP] \searrow \FF). $$
  \item Dual of (b). \qedhere
 \end{enumerate}
\end{proof}

\subsection{On some Frobenius subcategories of exact categories}

When we have an admissible monomorphism $f:X\to Y$ in an exact category, we say that $X$ is an \emph{admissible subobject} of $Y$. Dually we define an \emph{admissible factor object}.
For a full subcategory $\EE'$ of an exact category $\EE$, we denote by $\Sub\EE'$ the smallest full subcategory of $\EE$ which is closed under direct sums and admissible subobjects and contains $\EE'$.

We recall that an exact category is \emph{Frobenius} if it has enough injective objects, enough projective objects and they coincide. This subsection is devoted to prove the following result.

\begin{proposition} \label{frobprop}
Let $\EE$ be an exact category which has enough projective objects and enough injective objects.
Let $\UU$ be a subcategory of injective objects in $\EE$ satisfying $\UU=\add\UU$, and $\DD := \Sub \UU$.
Assume that projective objects of $\EE$ and those of $\DD$ coincide.
Then the following assertions hold.
\begin{enumerate}[\rm (a)]
 \item $\DD$ is closed under extensions.
 \item $\DD$ is Frobenius if and only if the following conditions are satisfied:
 \begin{itemize}
  \item $U$ is projective-injective in $\EE$ for any $U\in\UU$.
  \item Each projective object of $\EE$ has injective dimension at most $1$ and each injective object of $\EE$ has projective dimension at most $1$.
 \end{itemize}
 \item If the conditions in (b) are satisfied, then $\UU$ is the category of projective-injective objects in $\EE$.
 \end{enumerate}
\end{proposition}

Part (a) is an easy consequence of horseshoe lemma. Let us start by the following Lemma:

\begin{lemma}\label{BBI sequence}
Assume that any object in $\UU$ is projective in $\EE$. Let $0\rightarrow E\rightarrow E'\xrightarrow{f} I\rightarrow 0$
be an exact sequence in $\EE$ with $I$ injective.
Then $\Hom_\EE(\DD, f):\Hom_{\EE}(\DD,E')\rightarrow \Hom_{\EE}(\DD,I)$ is an epimorphism.
\end{lemma}

\begin{proof}
Take a morphism $g:D\rightarrow I$ with $D\in\DD$. Then there exists an admissible monomorphism $i:D\rightarrow U$ with $U\in\UU$.
Since $I$ is injective in $\EE$, there exists $s:U\rightarrow I$ such that $g=is$.
Since $U$ is projective in $\EE$, there exists $t:U\rightarrow E'$ such that $s=tf$.
\[\xymatrix{
0\ar[r]&E\ar[r]&E'\ar[r]^f&I\ar[r]&0\\
&&D\ar[ur]|(.3)g\ar[r]_i&U\ar[u]_{s}\ar[ul]|(.3)t
}\]
Since $g=itf$, we have the assertion.
\end{proof}

Let us now prove the proposition.

\begin{proof}[Proof of Proposition \ref{frobprop} (b) $\Rightarrow$]
Suppose that $\DD$ is Frobenius.
Note that our assumptions imply that projective objects in $\EE$, projective objects in $\DD$ and injective objects in $\DD$ coincide.

Fix any $U\in\UU$. Then $U$ is injective in $\EE$ by our assumption, and hence $U$ is injective also in $\DD$. Therefore $U$ is projective in $\EE$ by the remark above.

Let $P$ be a projective object in $\EE$. Then $P$ is projective-injective in $\DD$.
Since our assumptions imply $\Omega_{\EE}(\EE)\subset\DD$, we have
$\Ext^2_\EE(\EE, P) = \Ext^1_\EE(\Omega_\EE(\EE), P) = 0$.
Thus $P$ has injective dimension at most one in $\EE$.

 Let $I$ be an injective object in $\EE$. We take an exact sequence
\begin{equation}\label{Omega_BI}
0 \rightarrow \Omega_\EE(I) \rightarrow P \xrightarrow{f} I \rightarrow 0
\end{equation}
with a projective object $P$ in $\EE$. Our assumptions imply $P\in\DD$ and $\Omega_{\EE}(I)\in\DD$.
We apply $\Hom_\EE(\DD, -)$ to \eqref{Omega_BI} to get the exact sequence
 $$\Hom_\EE(\DD, P) \rightarrow \Hom_\EE(\DD, I) \rightarrow \Ext^1_\EE(\DD, \Omega_\EE(I)) \rightarrow \Ext^1_\EE(\DD, P) = 0.$$
By Lemma \ref{BBI sequence}, we have $\Ext^1_\EE(\DD, \Omega_\EE(I)) = 0$. Thus $\Omega_\EE (I)$ is projective-injective in $\DD$ so projective in $\EE$, and the assertion follows. 
\end{proof}

\begin{proof}[Proof of Proposition \ref{frobprop} (b) $\Leftarrow$]
 Let $P$ be a projective object in $\DD$. By our assumptions, $P$ is projective in $\EE$, and there exists an exact sequence
$0 \rightarrow P \rightarrow I^0 \rightarrow I^1 \rightarrow 0$
with injective objects $I^0$, $I^1$ in $\EE$.
Applying $\Hom_{\EE}(\DD,-)$, we have an exact sequence
$$\Hom_{\EE}(\DD,I^0)\rightarrow\Hom_{\EE}(\DD,I^1)\rightarrow \Ext^1_{\EE}(\DD,P)\rightarrow \Ext^1_{\EE}(\DD,I^0)=0.$$
By Lemma \ref{BBI sequence}, we have $\Ext^1_\EE(\DD, P) = 0$. Thus $P$ is injective in $\DD$. 

Let $I$ be an injective object in $\DD$. 
Since $\Omega_\EE(\EE)\subset\DD$, we have $\Ext^2_\EE(\EE, I) = \Ext^1_\EE(\Omega_\EE(\EE),I) = 0$.
Thus $I$ has injective dimension at most one in $\EE$.
Now we take an exact sequence
 \begin{equation}\label{IIB sequence}
0 \rightarrow I \rightarrow U \rightarrow E \rightarrow 0
 \end{equation}
 with $U\in\UU$ and $E \in \EE$. Since $U$ is injective in $\EE$, so is $E$.
Thus $E$ has projective dimension at most one in $\EE$. Since $U$ is projective in $\EE$, so is $I$.
Thus $I$ is projective in $\DD$.

 Since $\EE$ has enough projective objects and $\Omega_{\EE}(\EE)\subset\DD$ holds, $\DD$ also has enough projective objects.
It remains to prove that $\DD$ has enough injective objects.
Fix $D \in \DD$ and take an exact sequence
$0 \rightarrow D \rightarrow U \rightarrow E \rightarrow 0$ with $U\in\UU$ and $E\in\EE$.
Since $\EE$ has enough injective objects by our assumption, there exists an exact sequence $0\rightarrow E\rightarrow I\rightarrow E'\rightarrow 0$ with an injective object $I$ in $\EE$ and $E'\in\EE$.
Let $0 \rightarrow P_1 \rightarrow P_0 \rightarrow I \rightarrow 0$ be a projective resolution of $I$ in $\EE$. We have a commutative diagram of exact sequences:
 $$\xymatrix{
&&0\ar[d]&0\ar[d]\\
  0 \ar[r] & P_1 \ar[r] \ar@{=}[d] & X \ar[r] \ar[d] & E \ar[d] \ar[r] & 0 \\
  0 \ar[r] & P_1 \ar[r] & P_0\ar[d] \ar[r] & I \ar[r]\ar[d] & 0\\
&&E'\ar@{=}[r]\ar[d]&E'\ar[d]\\
&&0&0.
 }$$
Since $P_0\in\DD$, the middle column shows $X\in\DD$. On the other hand, we have the following commutative diagram of exact sequences:
 $$\xymatrix{
  & & 0 \ar[d] & 0 \ar[d] \\
  & & P_1 \ar[d] \ar@{=}[r]& P_1 \ar[d] \\
  0 \ar[r] & D \ar[r] \ar@{=}[d] & Y \ar[r] \ar[d] & X \ar[d] \ar[r] & 0 \\
  0 \ar[r] & D \ar[r] & U \ar[r] \ar[d]  & E \ar[r] \ar[d]& 0 \\
   & & 0 & 0.
 }$$
 As $P_1$ is projective-injective in $\DD$, the middle column splits and $Y \cong U\oplus P_1$ is injective in $\DD$.
The middle row gives an injective hull of $D$ in $\DD$. 
\end{proof}

\begin{proof}[Proof of Proposition \ref{frobprop} (c)]
Let $P$ be a projective-injective object in $\EE$.
Then it belongs to $\DD$, and there is a short exact sequence
 $0 \rightarrow P \rightarrow U \rightarrow E \rightarrow 0$
with $U\in\UU$ and $E\in\EE$. Since $P$ is injective in $\EE$, this sequence splits. Thus $P$ belongs to $\UU$.
\end{proof}

\section{Equivalences arising from torsion pairs on exact categories} \label{categ}

Throughout this section, we assume the following.
\begin{itemize}
\item $\EE$ is an exact category which is Krull-Schmidt.
\item $(\AA, \BB)$ is a \emph{torsion pair} of $\EE$, that is, the following conditions are satisfied:
\begin{itemize}
\item $\AA$ and $\BB$ are full subcategories of $\EE$ such that $\Hom_\EE(\AA, \BB) = 0$.
\item For any $E\in\EE$, there exists an exact sequence $0 \rightarrow A \rightarrow E \rightarrow B \rightarrow 0$ with $A \in \AA$ and $B \in \BB$.
\end{itemize}
\end{itemize}
Then $\AA$ is closed under taking extensions and admissible factor objects, and $\BB$ is closed under taking extensions and admissible subobjects.
On the other hand, the natural inclusion functor $\BB\rightarrow \EE$ has a left adjoint functor $F:\EE\rightarrow \BB$. This is dense and induces a dense functor
\[F:\EE/[\AA]\rightarrow \BB.\]

\subsection{Basic properties of $F:\EE/[\AA]\rightarrow \BB$}
We consider the full subcategories of $\EE$ defined by 
\begin{align*}
 \EE_1 &= \{X \in \EE \mid \Ext^1_\EE(X, \AA) = 0\}; \\ 
 \EE_2 &= \{X \in \EE \mid \Ext^1_\EE(X, \AA) = 0,\ \Ext^2_\EE(X, \AA) = 0\}.
\end{align*}
The subsection is devoted to prove the following result:
\begin{theorem} \label{categ1}
We have the following assertions.
  \begin{enumerate}[\rm (a)]
  \item The functor $F: \EE_1/[\AA] \rightarrow \BB$ is fully faithful.
  \item The essential image of $F: \EE_1/[\AA] \rightarrow \BB$ is the subcategory consisting of $B \in \BB$ such that $\Ext^1_\EE(B, \AA)$ is a finitely generated $\AA^{\op}$-module.
  \item If $\Ext^1_\EE(\AA, \BB) = 0$, then $F: \EE_2 \rightarrow \BB$ is exact bijective. 
  \item If any object in $\AA$ is projective in $\EE$, then $\EE_2/[\AA]$ inherits canonically the exact structure of $\EE_2$ and $F: \EE_2/[\AA] \rightarrow \BB$ is exact bijective. 
 \end{enumerate}
\end{theorem}

We denote by $T:\EE\rightarrow \AA$ the right adjoint functor of the inclusion functor $\AA\rightarrow \EE$. Then for any $E\in\EE$, there exists a short exact sequence 
 $$0 \rightarrow TE \xrightarrow{f} E \xrightarrow{g} FE \rightarrow 0$$
in $\EE$ with $TE\in\AA$ and $FE\in\BB$. Clearly $f$ is a right $\AA$-approximation and $g$ is a left $\BB$-approximation.

The proof of Theorem \ref{categ1} is divided into Lemmas \ref{DDcase}, \ref{fgn}, \ref{DDpcase} and \ref{DDpcasep}.

\begin{lemma} \label{DDcase}
 The functor $F: \EE_1 \rightarrow \BB$ induces a fully faithful functor $F: \EE_1/[\AA] \rightarrow \BB$. 
\end{lemma}

\begin{proof}
Fix $X, Y \in \EE_1$. Applying $\Hom_\EE(X, -)$ to the short exact sequence $0\rightarrow TY\rightarrow Y\rightarrow FY\rightarrow 0$, we obtain the short exact sequence
 $$0 \rightarrow \Hom_\EE(X, TY) \rightarrow \Hom_\EE(X, Y) \rightarrow \Hom_\EE(X, FY) \rightarrow \Ext^1_\EE(X,TY)=0$$
 the last equality follows from $X \in \EE_1$. So
 \begin{align*}
  \Hom_\EE(X, FY) & \cong \frac{\Hom_\EE(X, Y)}{ \Hom_\EE(X, TY) } = \Hom_{\EE/[\AA]}(X, Y)
 \end{align*}
 where we use the fact that the first arrow of $TY\rightarrow Y$ is a right $\AA$-approximation. 
 On the other hand, using adjunction we have an isomorphim 
 $\Hom_\EE(FX, FY) \cong \Hom_\EE(X, FY)$. Thus the assertion follows.
\end{proof}

Next we prove the following observation.

\begin{proposition} \label{fgn}
The following conditions are equivalent:
\begin{enumerate}[\rm (i)]
\item $F:\EE_1\rightarrow \BB$ is dense.
\item $\Ext^1_\EE(B, \AA)$ is a finitely generated $\AA^{\op}$-module for any $B\in\BB$.
\end{enumerate}
\end{proposition}

This follows immediately from the following result for Krull-Schmidt exact categories, which is a generalization of \cite[Proposition 1.4]{AuRe91}.

\begin{lemma}\label{AR generalization}
Let $\XX$ be a Krull-Schmidt exact category, and $\YY$ a subcategory of $\XX$ which is closed under extensions and direct summands. For $X \in \XX$, the following conditions are equivalent: 
 \begin{enumerate}[\rm (i)]
  \item There exists an exact sequence $0\rightarrow Y\rightarrow Z \rightarrow X\rightarrow 0$ with $Y\in\YY$ and $\Ext^1_{\XX}(Z,\YY)=0$.
  \item $\Ext^1_\XX(X, \YY)$ is finitely generated $\YY^{\op}$-module.
 \end{enumerate}
\end{lemma}

We include a proof for the convenience of the reader.
\begin{proof}
(i)$\Rightarrow$(ii) Applying $\Hom_\XX(-, \YY)$ to a short exact sequence $0 \rightarrow Y \rightarrow Z \rightarrow X \rightarrow 0$, we obtain the exact sequence
 $$\Hom_\XX(Y, \YY) \rightarrow \Ext^1_\XX(X, \YY) \rightarrow \Ext^1_\XX(Z, \YY) = 0.$$
Thus $\Ext^1_\XX(X, \YY)$ is a finitely generated $\YY^{\op}$-module.

(ii)$\Rightarrow$(i) Since $\YY$ is Krull-Schmidt, there exists a projective cover $\varphi: \Hom_\XX(Y, \YY) \rightarrow \Ext^1_\XX(X, \YY)$. Let 
  $$0 \rightarrow Y \xrightarrow{f} Z \xrightarrow{g} X \rightarrow 0$$
be a short exact sequence represented by $\varphi(\id_Y)\in\Ext^1_\XX(X, Y)$. 
Since $\varphi$ is right minimal, $f$ belongs to ${\rm rad} \XX$, and hence $g$ is right minimal.
To prove $\Ext^1_{\XX}(Z ,\YY)=0$, it suffices to show that any exact sequence
\begin{equation}\label{Ext(X,A')}
0 \rightarrow Y' \rightarrow W \xrightarrow{s} Z \rightarrow 0
\end{equation}
with $Y' \in \YY$ splits. We have the following commutative diagram of exact sequences:
  $$\xymatrix{
  & 0 \ar[d] & 0 \ar[d] \\
  & Y' \ar@{=}[r] \ar[d] & Y' \ar[d] \\
 0 \ar[r] & Y'' \ar[r] \ar[d] & W  \ar[r] \ar[d]^s & X  \ar[r] \ar@{=}[d] & 0 \\
 0 \ar[r] & Y \ar[r]^f \ar[d] & Z \ar[r]^g \ar[d] & X \ar[r] & 0 \\
 & 0 & 0
 }$$
 where $Y'' \in \YY$ because $\YY$ is extension-closed. 
 As $\varphi$ is an epimorphism, we have the following commutative diagram of exact sequences:
 $$\xymatrix{
 0 \ar[r] & Y \ar[r]^f \ar[d] & Z \ar[r]^g \ar[d]^t & X \ar[r] \ar@{=}[d] & 0 \\
 0 \ar[r] & Y'' \ar[r] & W \ar[r]  & X \ar[r] & 0 \\
 }$$
 As $g$ is right minimal, $ts : Z \rightarrow Z$ is invertible. Therefore the sequence \eqref{Ext(X,A')} splits.
\end{proof}

\begin{lemma} \label{DDpcase}
 Suppose that $\Ext^1_\EE(\AA, \BB) = 0$. Then the functor $F: \EE_2 \rightarrow \BB$ is exact bijective.
\end{lemma}

\begin{proof}
 Let $X, Y \in \EE_2$. Applying $\Hom_\EE(X, -)$ to the short exact sequence $0\rightarrow TY\rightarrow Y\rightarrow FY\rightarrow 0$, we have the isomorphism
 $$\Ext^1_\EE(X, Y) \cong \Ext^1_\EE(X, FY) $$
 as $\Ext^i_\EE(X,TY) = 0$ holds for $i=1,2$. Applying $\Hom_\EE(-, FY)$ to  the short exact sequence $0\rightarrow TX\rightarrow X\rightarrow FX\rightarrow 0$, we have an isomorphism
 $$\Ext^1_\EE(FX, FY) \cong \Ext^1_\EE(X, FY)$$
 as $\Ext^i_\EE(TX, FY) = 0$ holds for $i=0,1$. Thus we have
 \begin{align*} &\Ext^1_\BB(FX, FY) =\Ext^1_\EE(FX,FY) \cong \Ext^1_\EE(X,Y). \qedhere \end{align*}
\end{proof}

\begin{lemma} \label{DDpcasep}
 Suppose that any object in $\AA$ is projective in $\EE$. Then $\EE_2/[\AA]$ inherits canonically the exact structure of $\EE_2$, and the functor $F: \EE_2/[\AA] \rightarrow \BB$ is exact bijective.
\end{lemma}

\begin{proof}
 Any object $X\in\EE$ has a right $\AA$-approximation $TX\to X$ which is a categorical monomorphism, and we have $\Ext^1_{\EE}(\AA,\EE_2)=\Ext^1_{\EE}(\EE_2,\AA)=0$ by our assumptions. Therefore Corollary \ref{cdexactm} \eqref{cdexact} gives an exact structure on $\EE_2/[\AA]$.
Applying Lemma \ref{DDpcase}, we have $\Ext^1_\BB(FX, FY) \cong \Ext^1_{\EE_2}(X,Y)=\Ext^1_{\EE_2/[\AA]}(X, Y) $, which shows the assertion.
\end{proof}

\subsection{When there is a torsion pair $(\BB, \CC)$} \label{torsBC}

In this subsection, we assume the following.
\begin{itemize}
\item $(\BB, \CC)$ is a torsion pair in $\EE$ for
\[\CC  := \{X \in \EE \mid\Hom_\EE(\BB, X) = 0\}.\]
\end{itemize}
The following result gives a description of the image of the functor $F:\CC\rightarrow \BB$.

\begin{theorem} \label{abelian}
Assume that the following conditions are satisfied.
\begin{itemize}
\item $\BB$ is an abelian category whose exact structure is compatible with that of $\EE$.
\item $\Ext^1_\EE(\BB, A)$ is a finitely generated $\BB$-module for any $A\in\AA\cap\CC$.
\end{itemize}
Then we have the following assertions.
\begin{enumerate}[\rm (a)]
\item For any $A\in\AA\cap\CC$, there exists a short exact sequence
 $$0 \rightarrow A \rightarrow C^A \rightarrow U^A \rightarrow 0$$
 with $U^A\in\BB$, $C^A \in \CC$ and $\Ext^1_\EE(\BB, C^A) = 0$. Moreover, it is unique up to isomorphism.
\item Let $\DD:=\Sub\{U^A\mid A\in\AA\cap\CC\}$. Then $F:\EE\rightarrow \BB$ induces a dense functor $F: \CC\rightarrow\DD$.
\end{enumerate}
Assume $\Ext^2_\EE(\BB, \AA\cap\CC)=0$.
\begin{enumerate}[\rm (a)] \rs{2}
\item $U^A$ is an injective object in $\BB$ for any $A\in\AA\cap\CC$.
\item $\DD$ is closed under taking extensions in $\EE$, and therefore forms an exact category.
\item Assume $\CC\subset\EE_2$ and that any object in $\AA$ is projective in $\EE$. Then $\CC/[\AA]$ inherits canonically the exact structure of $\CC$ and
$F: \CC/[\AA] \rightarrow \DD$ is an equivalence of exact categories.
\end{enumerate}
\end{theorem}

\begin{proof}
(a) By dual of Lemma \ref{AR generalization}, we get a short exact sequence
 $$0 \rightarrow A \xrightarrow{f} X \xrightarrow{g} U^A \rightarrow 0$$
 for some $U^A \in \BB$ such that $\Ext^1_\EE(\BB, X) = 0$ and with $f$ left minimal. We only have to prove $X\in\CC$. Since $(\BB, \CC)$ is a torsion pair, there exists an exact sequence $$0\rightarrow B\xrightarrow{i} X\rightarrow C\rightarrow 0$$ with $B\in\BB$ and $C\in\CC$.
Now we consider the following commutative diagram, where $\Kernel ig$ exists in $\BB$ by our assumption.
$$\xymatrix{
0\ar[r]&A\ar[r]&X\ar[r]^g&U^A\ar[r]&0\\
0\ar[r]&\Kernel ig\ar[r]\ar[u]&B\ar[r]^{ig}\ar[u]^i&U^A\ar@{=}[u]
}$$
Since $A\in\CC$, we have $\Kernel ig=0$. Thus $ig$ is a monomorphism, and we can form the following commutative diagram with $\Cokernel ig \in \BB$
 by our assumption: 
 $$\xymatrix{
 & & 0 \ar[d] & 0 \ar[d] \\
 & & B \ar@{=}[r] \ar[d]^i & B \ar[d]^{ig} \\
 0 \ar[r] & A \ar[r]^f \ar@{=}[d] & X \ar[r]^g \ar[d] & U^A \ar[r] \ar[d]^p & 0 \\
 0 \ar[r] & A \ar[r] & C \ar[r] \ar[d] & \Cokernel ig \ar[r] \ar[d] & 0 \\
  & & 0 & 0 
 }$$
The upper horizontal sequence gives a projective cover $\varphi: \Hom_{\EE}(\BB,U^A)\rightarrow \Ext^1_{\EE}(\BB,A)$ (see Proof of Lemma \ref{AR generalization}). The lower horizontal sequence gives a morphism $\psi:\Hom_{\EE}(\BB,\Cokernel ig)\rightarrow \Ext^1_{\EE}(\BB,A)$, which is an epimorphism since $\varphi = \Hom_\EE(\BB, p) \psi$. Since $\Cokernel ig\in\BB$ and $\varphi$ is a projective cover, $p$ has to be an isomorphism. Thus we have $B=0$ and $X\cong C\in\CC$. 

As $\BB \cap \CC = 0$, the morphism $A \to C^A$ is left minimal and it implies easily the uniqueness.

(b) First we prove $F(\CC)\subset\DD$.
For any $C\in\CC$, there exists an exact sequence $0 \rightarrow A \rightarrow C \rightarrow B \rightarrow 0$
 with $B=FC\in\BB$ and $A=TC \in \AA$.
Clearly we have $A\in\AA\cap\CC$. Let $0\rightarrow A\rightarrow C^A\rightarrow U^A\rightarrow 0$ be the exact sequence in (a). Then we have a commutative diagram
 $$\xymatrix{
  0 \ar[r] & A \ar[r] \ar@{=}[d] & C \ar[r] \ar[d] & B \ar[r] \ar[d]^f & 0 \\
  0 \ar[r] & A \ar[r] & C^A \ar[r] & U^A \ar[r] & 0.
 }$$
By our assumption, $f$ has a kernel $g:\Kernel f\rightarrow B$ in $\EE$ with $\Kernel f\in\BB$.
Since the above diagram is pullback, $g$ factors through $C\in\CC$. Thus $g=0$ holds, and hence $f$ is a monomorphism. 
Therefore $0\rightarrow B\xrightarrow{f}U^A\rightarrow \Cokernel f\rightarrow 0$ is a short exact sequence in $\EE$ by our assumption, and $B\in\DD$ holds. 

Next we prove that the functor $F:\BB\to\DD$ is dense.
For any $D \in \DD$, there exist exact sequences
\[0\rightarrow D\rightarrow U^A\rightarrow X\rightarrow 0\ \mbox{ and }\ 0\rightarrow A\rightarrow C^A\rightarrow U^A\rightarrow 0\]
with $A\in\AA\cap\CC$, $U^A\in\BB$,  $C^A\in\CC$ and $\Ext^1_{\EE}(\BB,C^A)=0$. Then we have a commutative diagram
 $$\xymatrix{
&&0\ar[d]&0\ar[d]\\
  0 \ar[r] & A \ar[r] \ar@{=}[d] & Y \ar[r] \ar[d] & D \ar[r] \ar[d] & 0 \\
  0 \ar[r] & A \ar[r] & C^A \ar[r]\ar[d] & U^A \ar[r]\ar[d] & 0 \\
&&X\ar@{=}[r]\ar[d]&X\ar[d]\\
&&0&0
 }$$
of exact sequences. Since $C^A\in\CC$, we have $Y\in\CC$ by the middle vertical sequence. Therefore $D=FY$ belongs to $F(\CC)$.

(c) Applying $\Hom_\EE(\BB, -)$ to the short exact sequence $0\rightarrow A\rightarrow C^A\rightarrow U^A\rightarrow 0$, we have an exact sequence
\[0=\Ext^1_{\EE}(\BB,C^A)\rightarrow \Ext^1_{\EE}(\BB,U^A)\rightarrow \Ext^2_{\EE}(\BB,A)=0.\]
Therefore $\Ext^1_{\EE}(\BB,U^A)=0$, that is, $U^A$ is injective in $\BB$.

(d) This is an immediate consequence of (c) and horseshoe lemma.

(e) By Theorem \ref{categ1} (a) and (d), the functor $F: \EE_2/[\AA] \rightarrow \BB$ is fully faithful and exact bijective.
By $\CC\subset\EE_2$, using (b) and (d), we have an equivalence $F: \CC/[\AA] \rightarrow\DD$ of exact categories.
\end{proof}

\subsection{Frobenius properties} 

As in Subsection \ref{torsBC}, we suppose that $(\BB, \CC)$ is a torsion pair. We denote $\UU := \add \{U^A\mid A \in \AA \cap \CC\}$ and as in Theorem \ref{abelian}, $\DD := \Sub \UU$. The following result gives a sufficient condition for the categories $\CC$ and $\DD$ to be Frobenius.

\begin{theorem} \label{frob}
Assume that the following conditions are satisfied.
 \begin{itemize}
\item $\BB$ is an abelian category whose exact structure is compatible with that of $\EE$, and has enough projective objects and enough injective objects.
 \item $\AA\subset\CC$ holds, and any object in $\AA$ is projective in $\EE$ and injective in $\CC$.
\item $\Ext^1_\EE(\BB, A)$ is a finitely generated $\BB$-module for any $A\in\AA$.
   \item $\Ext^1_\EE(P, \AA)$ is a finitely generated $\AA^{\op}$-module for any projective object $P$ in $\BB$.
 \end{itemize}
Then we have the following assertions.
 \begin{enumerate}[\rm (a)] 
  \item The following conditions are equivalent:
\begin{enumerate}[\rm (i)]
\item Projective objects of $\BB$ and $\DD$ coincide.
\item Projective objects of $\CC$ and $\EE$ coincide.
\end{enumerate}
 \end{enumerate}
 Suppose that the equivalent conditions in (a) are satisfied. Then the following assertions hold.
 \begin{enumerate}[\rm (a)] \rs{1}
  \item $\EE$ and $\CC$ have enough projective objects.
  \item Any object in $\AA$ has injective dimension at most $1$ in $\EE$. Therefore all assertions in Theorem \ref{abelian} hold.
  \item The following conditions are equivalent:
 \begin{enumerate}[\rm (i)]
  \item $\CC$ is a Frobenius category whose exact structure is compatible with that of $\EE$.
  \item $\DD$ is a Frobenius category whose exact structure is compatible with that of $\EE$.
  \item Any object in $\UU$ is projective-injective in $\BB$. Moreover each projective object of $\BB$ has injective dimension at most $1$ and each injective object of $\BB$ has projective dimension at most $1$.
 \end{enumerate}
 \item If the conditions in (d) are satisfied, then the category of projective-injective objects in $\BB$ is $\UU$.
 \end{enumerate}
\end{theorem}

We start with preparing the following:

\begin{lemma}\label{cover of P}
For any projective object $P$ in $\BB$, there exists a projective object $X$ in $\EE$ such that $P=FX$.
\end{lemma}

\begin{proof}
Thanks to Lemma \ref{AR generalization}, there exists a short exact sequence
 $0 \rightarrow A \rightarrow X \rightarrow P \rightarrow 0$
with $A \in \AA$ and $\Ext^1_{\EE}(X,\AA)=0$.
Applying $\Hom_\EE(-, \BB)$, we have an exact sequence
\[0=\Ext^1_{\EE}(P,\BB)\rightarrow \Ext^1_{\EE}(X,\BB)\rightarrow \Ext^1_{\EE}(A,\BB)=0.\]
Thus $\Ext^1_\EE(X, \BB) = 0$ holds. Since $\Ext^1_{\EE}(X,\AA)=0$, we have $\Ext^1_{\EE}(X,\EE)=0$.
Thus $X$ is a projective object in $\EE$ satisfying $P=FX$.
\end{proof}

Now we are ready to prove Theorem \ref{frob}.

\begin{proof}[Proof of Theorem \ref{frob}]
(a) (ii)$\Rightarrow$(i) Suppose that projective objects of $\CC$ and $\EE$ coincide. 

Let $P$ be a projective object in $\BB$. By Lemma \ref{cover of P}, there exists a projective object $X$ in $\EE$ such that $P=FX$.
Since $X$ belongs to $\CC$ by our assumption, we have $P\in F(\CC)\subset\DD$.

Let $P$ be a projective object in $\DD$. Since $\BB$ has enough projective objects by our assumption, there exists a projective cover $f:X\rightarrow P$ in $\BB$. Since $X$ belongs to $\DD$ by the above argument, $f$ splits. Thus $P$ is projective in $\BB$.

(i)$\Rightarrow$(ii) Suppose that projective objects of $\BB$ and $\DD$ coincide. 

Let $P$ be a projective object in $\EE$. 
Let $0\rightarrow X\rightarrow P'\xrightarrow{f} FP\rightarrow 0$ be an exact sequence with a projective object $P'$ in $\BB$. Then $P'\in\DD$ by our assumption. By Theorem \ref{abelian} (b), there exists an exact sequence $0 \rightarrow A \xrightarrow{i} C \xrightarrow{p} P' \rightarrow 0$
 with $A \in \AA$ and $C \in \CC$. Since $\Ext^1_{\EE}(C,\AA)=0$ holds by our assumption, we have a commutative diagram:
$$\xymatrix{
  0 \ar[r] & A\oplus TP \ar[d]_{\begin{sbmatrix} \alpha \\ 1_{TP} \end{sbmatrix}} \ar[r]^{\begin{sbmatrix} i & 0 \\ 0 & 1_{TP} \end{sbmatrix}} & C\oplus TP \ar[d]^{\begin{sbmatrix} \beta \\ u \end{sbmatrix}} \ar[r]^-{\begin{sbmatrix} p \\ 0 \end{sbmatrix}} & P' \ar[d]^f \ar[r] & 0 \\
  0 \ar[r] & TP \ar[r]_u & P \ar[r]_v & FP \ar[r] & 0.
 }$$
 As $\begin{sbmatrix} \alpha \\ 1_{TP}  \end{sbmatrix}$ and $f$ are (admissible) epimorphisms, then $\begin{sbmatrix} \beta \\ u  \end{sbmatrix}$ is also one. Since $P$ is projective in $\EE$, $\begin{sbmatrix} \beta \\ u  \end{sbmatrix}$ splits. Thus $P$ is a direct summand of $C\oplus TP$, which belongs to $\CC$ by our assumption $TP\in\AA\subset\CC$.

 Conversely, let $Q$ be a projective object in $\CC$. Let us consider its projective cover $P$ in $\EE$. We get the short exact sequence
 $$0 \rightarrow \Omega_\EE Q \rightarrow P \rightarrow Q \rightarrow 0.$$
 According to the previous discussion, $P \in \CC$. Thus, applying $\Hom_\EE(\BB, -)$ to this short exact sequence, we get that $\Omega_\EE Q \in \CC$. Hence, as $Q$ is projective in $\CC$, the short exact sequence splits and $Q$ is projective in $\EE$.
 
(b) We now suppose that the conditions in (a) are satisfied.

For any $X\in\EE$, there exists a short exact sequence
 $0 \rightarrow A \rightarrow X \rightarrow B \rightarrow 0$
with $A=TX \in \AA$ and $B=FX \in \BB$. Then $A$ is projective in $\EE$ by our assumption.
Thus, to show that $X$ has a projective cover in $\EE$, it suffices to show that any $B\in\BB$ has a projective cover in $\EE$.

By our assumption, there exists a projective cover $f:P\rightarrow B$ in $\BB$.
By Lemma \ref{cover of P}, there exists a projective cover $g:P'\rightarrow P$ in $\EE$. Then the composition $gf:P'\rightarrow B$ gives a projective cover of $B$ in $\EE$.

(c) All projective objects of $\EE$ belong to $\CC$ by our assumption.
Therefore $\Omega_{\EE}(\EE)\subset\CC$ holds. Since any object in $\AA$ is injective in $\CC$ by our assumption, we have
\[\Ext^2_\EE(\EE, \AA) = \Ext^1_\EE(\Omega_\EE(\EE), \AA) = 0.\]
Thus the first assertion follows.
In particular we have $\Ext^2_{\EE}(\BB,\AA\cap\CC)=0$, and the second assertion follows.

(d) and (e) Thanks to Theorems \ref{eq1} and \ref{abelian} (e), $F: \CC \to \CC/[\AA] \to \DD$ is exact bijective. So $\CC$ is Frobenius if and only if $\DD$ is Frobenius by Remark \ref{remfrob}. Hence (i)$\Leftrightarrow$(ii) in (d) is proven. The remaining assertions follow from Proposition \ref{frobprop}.
\end{proof}

\section{Equivalences arising from orders and their idempotents} \label{ss:pforders}

As in Subsection \ref{ss:orders}, let $R$ be a complete discrete valuation ring and $K$ be its field of fractions. Fix an $R$-order $A$.
Consider functors $D_i:=\Ext^{1-i}_R(-,R):\mod A\leftrightarrow\mod A^{\op}$
for $i=0,1$. They restrict to dualities
\[D_1=\Hom_R(-,R):\CM A\stackrel{\sim}{\longleftrightarrow}\CM A^{\op}\ \mbox{ and }\ 
D_0=\Ext^1_R(-,R):\fl A\stackrel{\sim}{\longleftrightarrow}\fl A^{\op}\]
and satisfy $D_0(\CM A) = D_1(\fl A) = 0$. In view of the characterisations of $\CM A$ given at the beginning of Section \ref{s:orders}, it is immediate that $\CM A$ admits the projective generator $A$ and the injective co-generator $D_1 A$. 
Since the injective resolution of the $R$-module $R$ is given by $0 \rightarrow R \rightarrow K \rightarrow K/R \rightarrow 0$, we get an isomorphism $D_0 \cong \Hom_R(-,K/R)$ on $\fl A$.
Recall the following useful lemma:
\begin{lemma} \label{injfl}
 If $X \in \CM A$, we have a monomorphism $X \hookrightarrow X \otimes_R K$ and $\Ext^1_A(\fl A, X \otimes_R K) = 0$.
\end{lemma}

\begin{proof}
For $Y\in\fl A$, let $E:=\Ext^1_A(Y, X \otimes_R K)$. Since $X\otimes_RK$ is a $K$-vector space, so is $E$. Since $Y$ is annihilated by some non-zero element in $R$, so is $E$. These imply $E=0$.
\end{proof}

For an object $X \in \CM A$, let $\corad X \in \CM A$ be maximal among $A$-submodules $Y$ of $X \otimes_R K$ such that $X \subset Y$ and $Y/X$ is semisimple. We denote  $\cotop X := (\corad X) / X$. Notice that $X \otimes_R K$ is not finitely generated as an $A$-module (so $X \otimes_R K \notin \CM A$) if $X \in \CM A$ is non-zero. Notice also that $D_1 (X \otimes_R K) = 0$.

We often use the following lemma:
\begin{lemma} Let $X \in \CM A$. The following hold: \label{generalitiescotop}
 \begin{enumerate}[\rm (a)] 
  \item We have $\cotop X = \soc (X \otimes_R (K/R) )$.
  \item The functor $D_1$ induces an order-reversing bijection
   $$\{X \subset Y \subset X \otimes_R K \mid Y/X \in \fl A\} \xleftrightarrow{1 - 1} \{Y' \subset D_1 X \mid (D_1 X) / Y' \in \fl A^{\op} \}.$$
  \item There are isomorphisms $\corad X \cong D_1 \rad D_1 X$ and $\cotop X \cong D_0 \top D_1 X$ of $A$-modules.
  \item If $0 \rightarrow X \rightarrow Y \rightarrow S \rightarrow 0$ is a short exact sequence with $Y \in \CM A$ and a semisimple $A$-module $S$, then there is a unique canonical commutative diagram
  $$\xymatrix{
   0 \ar[r] & X \ar[r] \ar@{=}[d] & Y \ar[r] \ar@{^{(}->}[d] & S \ar[r] \ar@{^{(}->}[d] & 0 \\
   0 \ar[r] & X \ar[r] & \corad X \ar[r] & \cotop X \ar[r] & 0.}$$
  \item For a simple $A$-module $S$, we have $\Ext^1_A(S, X) \neq 0$ if and only if $S$ is a direct summand of $\cotop X$.
 \end{enumerate}
\end{lemma}

\begin{proof}
 (a) and (b) are immediate and the first isomorphism of (c) is a consequence of (b). The second isomorphism of (c) is obtained by applying $\Hom_R(-,R)$ to the short exact sequence $0 \to \rad D_1 X \to D_1 X \to \top D_1 X \to 0$. For (d), applying the functor $- \otimes_R K$ to the short exact sequence we get $X \otimes_R K \cong Y \otimes_R K$. Therefore $X \subset Y \subset X \otimes_R K$. By maximality of $\corad X$, we have $Y \subset \corad X$ and the result follows. 

 (e) The implication $\Leftarrow$ is immediate. Let us show $\Rightarrow$. Consider a non-split exact sequence $0 \to X \to Y \to S \to 0$. For any simple module $S'$, applying $\Hom_A(S', -)$, we get an exact sequence $0 \to \Hom_A(S', Y) \to \Hom_A(S', S) \to \Ext^1_A(S',X)$. It is easy to conclude in any case that $\Hom_A(S', Y) = 0$, so $Y \in \CM A$. Therefore, we can apply (d) so $S$ is a summand of $\cotop X$. 
%
%
%
%
\end{proof}

 For logical reasons, we give Proof of Theorem \ref{mainA} after Proof of Theorem \ref{mainB}.

\subsection{Proof of Theorem \ref{mainB}}
 As in Theorem \ref{mainB}, we consider an idempotent $e$ of an $R$-order $A$ such that that $B := A / (e)$ has finite length over $R$. As $\mod B \subset \fl A$, $D_0$ restricts to a duality $\mod B\stackrel{\sim}{\longleftrightarrow}\mod B^{\op}.$
 We will separate the proof in five statements.
 \begin{proposition}\label{pair1}
  We have a torsion pair $(\add Ae, \mod B)$ in $\mod_e A$.
 \end{proposition}
 
 \begin{proof}
  Since $B=A/(e)$, we have $\Hom_A(\add Ae, \mod B) = 0$.
For any $X \in \mod_e A$, we have an exact sequence 
\begin{equation}\label{sequence for X}
Ae \otimes_{e A e} e X \xrightarrow{f} X \rightarrow B \otimes_A X \rightarrow 0
\end{equation}
 in $\mod_e A$. Since $e X \in \proj (e A e)$, we have $A e \otimes_{e A e} e X \in \add Ae$.
Multiplying the sequence \eqref{sequence for X} by $e$ on the left, we see that $e \Kernel f = 0$ so $\Kernel f$ is in $\mod B$.
On the other hand, $\Kernel f$ is a submodule of $A e \otimes_{e A e} e X\in\add Ae$, so $\Kernel f \in \CM A$.
Consequently we have $\Kernel f = 0$.
Now the sequence \eqref{sequence for X} shows the desired assertion.
 \end{proof}

Thanks to Proposition \ref{pair1}, we have two functors $T: \mod_e A \to \add Ae$ and $F: \mod_e A \to \mod B$ and a functorial exact sequence
$0 \to TX \to X \to FX \to 0$
for $X \in \mod_e A$. We prove the following easy statement:

\begin{lemma} \label{LemmaFX}
 If $X \in \CM_e A$, then $FX \subset \Hom_A(B, TX \otimes_R (K/R)) \subset TX \otimes_R (K/R)$.
\end{lemma} 

\begin{proof}
 The inclusion $\Hom_A (B, TX \otimes_R (K/R)) \subset TX \otimes_R (K/R)$ is obvious. Applying $- \otimes_R K$ on the short exact sequence $0 \to TX \to X \to FX \to 0$, we get that $TX \otimes_R K \cong X \otimes_R K$ so $X \subset TX \otimes_R K$ canonically. Thus we get a commutative diagram of short exact sequences
 $$\xymatrix{
  0 \ar[r] & TX \ar[r] \ar@{=}[d] & X \ar[r] \ar@{^{(}->}[d] & FX \ar[r] \ar@{^{(}->}[d]& 0 \\
  0 \ar[r] & TX \ar[r] & TX \otimes_R K \ar[r] & TX \otimes_R (K/R) \ar[r] & 0
 }$$
 where the second line is obtained by applying $TX \otimes_R -$ to $0 \to R \to K \to K/R \to 0$. Thus $FX \subset TX \otimes_R (K/R)$. As $FX \in \mod B$, we deduce that $FX \subset \Hom_A (B, TX \otimes_R (K/R))$.
\end{proof}

 \begin{proposition}\label{pair2}
  We have a torsion pair $(\mod B, \CM_e A)$ in $\mod_e A$.
 \end{proposition}

 \begin{proof}
  Since any $X\in\mod B$ has finite length, we have $\Hom_A(\mod B, \CM_e A) = 0$.
For any $X \in \mod_e A$, there exists an exact sequence
 $$0 \rightarrow T \rightarrow X \rightarrow F \rightarrow 0$$
in $\mod A$ such that $\length_RT<\infty$ and $F\in\CM A$.
Multiplying $e$ from the left, we have an exact sequence
\[0\rightarrow eT\rightarrow eX\rightarrow eF\rightarrow 0\]
with $\length_R(eT)<\infty$ and $eX\in\proj(eAe)$. Thus $eT=0$ holds, and we have $T\in\mod B$.
On the other hand, $eF=eX\in\proj(eAe)$ shows $F\in\CM_eA$.
Thus the assertion follows.
 \end{proof}

Now we can apply Theorems \ref{categ1} and \ref{abelian} to 
\[\EE:=\mod_eA,\ \AA:=\add Ae,\ \BB:=\mod B\ \mbox{ and } \CC:=\CM_eA.\]
In this context, it is possible to compute explicitly the short exact sequence given in Theorem \ref{abelian}~(a). For $P \in \add Ae$, let $$U^P:=\Hom_A(B,P\otimes_R(K/R))\in\mod B$$
and define $U := U^{Ae}$.
For any $X \in \CM A$, we denote $$\Bcotop X := \Hom_A(B, \cotop X).$$ In other terms, $\Bcotop X$ is the biggest $B$-module included in $\cotop X$. We also define $\Bcorad X$ as the $A$-module satisfying 
 $$X \subset \Bcorad X \subset \corad X \quad \text{and} \quad \Bcotop X \cong (\Bcorad X) / X.$$

 \begin{lemma}\label{U is universal} \label{socU}
 Let $P \in \add Ae$. The following hold:
 \begin{enumerate}[\rm (a)]
  \item There is a short exact sequence
$0\rightarrow P\rightarrow C^P\rightarrow U^P\rightarrow 0$
in $\mod A$ with $C^P\in\CM_eA$ and $\Ext^1_A(\mod B,C^P)=0$. 
  
  Conversely, if $0 \to P \to C' \to U' \to 0$ is a short exact sequence with $C' \in \CM_e A$, $U' \in \mod B$ and $\Ext^1_A(\mod B, C') = 0$, then it is isomorphic to the above short exact sequence.
  \item We have an isomorphism $\soc U^P \cong \Bcotop P$ of $B$-modules.
 \end{enumerate}
 \end{lemma}

 \begin{proof}
(a) Applying $P \otimes_R -$ to the short exact sequence $0 \rightarrow R \rightarrow K \rightarrow K/R \rightarrow 0$, we obtain the short exact sequence
$0 \rightarrow P \rightarrow P \otimes_R K \rightarrow P \otimes_R (K/R) \rightarrow 0$
with $\Ext^1_A(\fl A, P \otimes_R K) = 0$ thanks to Lemma \ref{injfl}. Taking pullback by the natural inclusion $U^P \subset P \otimes_R (K/R)$, we get the following commutative diagram of short exact sequences:
$$\xymatrix{
 & & 0 \ar[d] & 0 \ar[d] \\
 0 \ar[r] & P \ar[r] \ar@{=}[d] & C^P \ar[r] \ar[d] & U^P \ar[r] \ar[d] & 0 \\
 0 \ar[r] & P \ar[r]  & P \otimes_R K \ar[r] \ar[d] & P \otimes_R (K/R) \ar[r] \ar[d] & 0 \\
 & & Y \ar@{=}[r] \ar[d] & Y \ar[d] \\
 & & 0 & 0.
}$$
Since $U^P$ is the maximal $B$-module included in $P \otimes_R (K/R)$, and $\mod B$ is closed under extensions in $\mod A$, we get $\Hom_A(\mod B, Y) = 0$. Applying $\Hom_A(\mod B, -)$ to the second column, we find the exact sequence
 \begin{align*} &0 = \Hom_A(\mod B, Y) \rightarrow \Ext^1_A(\mod B, C^P) \rightarrow \Ext^1_A(\mod B, P \otimes_R K) = 0.  \end{align*}

 Now we prove the converse part. Applying $\Hom_A (U', -)$ to the former sequence, we get a surjection $\Hom_A(U', U^P) \twoheadrightarrow \Ext^1_A(U', P)$ so there is a
 commutative diagram
 $$\xymatrix{
  0 \ar[r] & P \ar@{=}[d] \ar[r] & C' \ar[d]^f \ar[r] & U' \ar[d]^g \ar[r] &  0 \\
  0 \ar[r] & P \ar[r] & C^P \ar[r] & U^P  \ar[r] & 0. 
 }$$
 In the same way, there are $f': C^P \to C'$ and $g': U^P \to U'$ making commutative diagram in the converse direction. Thus, by left minimality of $P \to C^P$ and $P \to C'$, $ff'$ and $f'f$ are isomorphisms. Hence, $f$ and $g$ are isomorphisms.

(b) Thanks to Lemma \ref{generalitiescotop}, $\cotop P = \soc (P \otimes_R (K/R))$. Applying $\Hom_A(B,-)$ to both sides, we obtain
 $\Hom_A(B, \cotop P) = \Hom_A(B, \soc (P \otimes_R (K/R)) = \soc U^P$.
 \end{proof}

We are ready to prove Theorem \ref{mainB}.

\begin{proof}[Proof of Theorem \ref{mainB}]
(a) This follows from Propositions \ref{pair1} and \ref{pair2}.

(b) This follows from Theorem \ref{categ1} (a) and (b) as, for $Y \in \mod B$, $\Ext^1_A(Y, Ae)$ is always a finitely generated right $(eAe)$-module.

(c) Our assumption (E1) implies $\CM_eA\subset\EE_1$.
Thus the functor $F:(\CM_eA)/[Ae]\rightarrow \mod B$ is fully faithful by (a).
It gives an equivalence $F:(\CM_eA)/[Ae]\rightarrow \Sub U$ by Theorem \ref{abelian} (b) and Lemma \ref{U is universal}. 

(d) It follows from Theorem \ref{abelian} (c). 

(e) Thanks to (E2), $\EE_1 = \EE_2$ so, using Theorem \ref{categ1} (d), \eqref{functor F} and \eqref{functor F2} are equivalences of exact categories. 

(f) It is classical that $\Sub U$ has enough projective objects and enough injective objects (see \cite{DeIy-2} for a detailed argument). Using (e) and Remark \ref{remfrob} (b), it immediately implies that $\CM_e A$ has enough projective objects and enough injective objects. In the same way, as $\mod B$ has enough injective objects and enough projective objects, $\EE_1$ has the same property. 

 Let us prove that $\mod_e A$ has enough projective objects. Let $X \in \mod_e A$. Thanks to (b) and (e), there exist $P_0 \in \EE_1$ projective in $\EE_1$ such that $FP_0$ is a projective cover of $FX$. Fixing $P := P_0 \oplus TX$, we get a short exact sequence $0 \to K \to P \to X \to 0$ where, multiplying by $e$, we have $K \in \mod_e A$. As $P \in \EE_1$, $\Ext^1_A(P, Ae) = 0$. As $\Ext^1_A(TP, \mod B) = 0 = \Ext^1_A(FP, \mod B)$, we get $\Ext^1_A(P, \mod B) = 0$ and, as $(\add Ae, \mod B)$ is a torsion pair in $\mod_e A$, $\Ext^1_A(P, \mod_e A) = 0$.

 Let us prove that $\mod_e A$ has enough injective objects. For any $I$ injective in $\mod B$, we have $\Ext^1_A(Ae, I) = 0 = \Ext^1_A(\mod B, I)$. So, as $(\add Ae, \mod B)$ is a torsion pair in $\mod_e A$ by Proposition \ref{pair1}, we get $\Ext^1_A(\mod_e A, I) = 0$ so any $B$-module admits an injective hull in $\mod_e A$. As $(\mod B, \CM_e A)$ is a torsion pair in $\mod_e A$ by proposition \ref{pair2}, it is enough, thanks to horseshoe Lemma, to prove that any object in $\CM_e A$ admits an injective hull in $\mod_e A$. 

 Consider $C := C^{Ae}$ as defined in Lemma \ref{socU}. As $\Ext^1_A(Ae, C) = 0 = \Ext^1_A(\mod B, C)$, we get $\Ext^1_A(\mod_e A, C) = 0$. In particular $0 \to Ae \to C \to U \to 0$ gives an injective hull of $Ae$ in $\mod_e A$. Let $X \in \CM_e A$. As $FX \in \Sub U$, consider a short exact sequence $0 \to FX \to U^n \to K_0 \to 0$ in $\mod B$. By (e), it is induced by a short exact sequence $0 \to X \to C^n \oplus P \to K \to 0$ in $\EE_1$ with $P \in \add Ae$. As $P$ admits an injective hull in $\mod_e A$, $X$ also admits an injective hull.

 (g) For any $X \in \mod_e A$, as $(\mod B, \CM_e A)$ is a torsion pair, there is a short exact sequence
 $$0 \rightarrow Z \rightarrow X \rightarrow Y \rightarrow 0$$
 where $Z \in \mod B$ and $Y \in \CM_e A$. Applying $\Hom_A(-, Ae)$ to this sequence, we find the exact sequence
 $$0 = \Ext^1_A(Y, Ae) \rightarrow \Ext^1_A(X, Ae) \rightarrow \Ext^1_A(Z, Ae) \rightarrow \Ext^2_A(Y, Ae) = 0$$
 so $X \in \EE_1$ if and only if $\Ext^1_A(Z, Ae) = 0$. There is a short exact sequence
 $$0 \rightarrow \soc Z \rightarrow Z \rightarrow Z/\soc Z \rightarrow 0$$
 and applying $\Hom_A(-, Ae)$ to it, we find the exact sequence
 $$0 \rightarrow \Ext^1_A(Z/\soc Z, Ae) \rightarrow \Ext^1_A(Z, Ae) \rightarrow \Ext^1_A(\soc Z, Ae) \rightarrow 0$$
 so $\Ext^1_A(Z, Ae) = 0$ if and only if $\Ext^1_A(Z/\soc Z, Ae) = \Ext^1_A(\soc Z, Ae) = 0$. By Lemma \ref{generalitiescotop} (e), for a simple $B$-module $S$, $\Ext^1_A(S, Ae) = 0$ if and only if $S$ is not a direct summand of $\Bcotop Ae$ if and only if $S \notin \Sub U$ if and only if $\Hom_A(P, S) = 0$ where $P$ is the projective cover of $\soc U$ in $\mod B$. As $Z$ is of finite length over $R$, an easy induction gives that $\Ext^1_A(Z, Ae) = 0$ if and only if $\Hom_A(P, Z) = 0$ if and only if $\Hom_A (P, X) = 0$.
\end{proof}

In the following Lemma, we give stronger conditions implying (E1) and (E2):
 \begin{lemma} \label{puprem}
  \begin{enumerate}[\rm (a)]
   \item We have the implications (E2)\pup* $\Rightarrow$ (E2).
   \item If $Ae=\Hom_R(gA,R)$ for some idempotent $g \in A$, then (E1) and (E2)\pup* are satisfied.
   \item If (E1) is satisfied and $A \in \CM_e A$, then (E2)\pup* is satisfied.
  \end{enumerate}
 \end{lemma}

 \begin{proof}
  (a) It directly follows from Proposition \ref{ext2inj}.

  (b) In this case, $\Ext^1_A(\CM A, Ae) = 0$ so (E1) is clearly satisfied. If $X \in \mod A$, it is immediate that its syzygy $\Omega X$ is in $\CM A$ so $\Ext^2_A(X, Ae) = \Ext^1_A(\Omega X, Ae) = 0$. Therefore, (E2)\pup* holds.

  (c) For $X \in \mod_e A$, consider the projective cover $0 \to \Omega X \to P \to X \to 0$.  As $e X \in \proj (e A e)$, the short exact sequence $0 \to e \Omega X \to e P \to e X \to 0$ splits. Moreover, as $A \in \CM_e A$, we have $e P \in \proj (e A e)$ so $\Omega X \in \CM_e A$. So, by (E1), $\Ext^2_A(X, Ae) = \Ext^1_A(\Omega X, Ae) = 0$ and (E2)\pup* holds.
 \end{proof}

We complete this subsection by giving basic relations between indecomposable injective objects of $\CM_e A$ and their $B$-cotops. Let $$\OO := \{P \in \ind Ae \,|\, \Bcotop P \neq 0\}.$$ Notice that part (a) of Lemma \ref{Bcotopsimple} is a generalization of a well-known property of cotops in $\CM A$.

\begin{lemma} \label{Bcotopsimple}
 Let $I \in \CM_e A$ satisfying $\Ext^1_A(\CM_e A, I) = 0$. Then the following hold:
 \begin{enumerate}[\rm (a)]
  \item if I is indecomposable, then $\Bcotop I$ is either $0$ or simple;
  \item $\Bcotop I = 0$ if and only if $\Ext^1_A(\mod B, I) = 0$;
  \item for any short exact sequence $0 \to I \xrightarrow{i} X \xrightarrow{p} Y \to 0$ with $i$ radical, $X \in \CM_e A$ and $Y \in \mod_e A$, the left map factors as $I \subset \Bcorad I \hookrightarrow X$ and $\soc Y \cong \Bcotop I$.
  \item if (E1) is satisfied, there are commuting bijections
  \begin{align*}&\xymatrix@R=1cm@C=2.5cm@!=0cm{
    \OO \ar[drr]_{\Bcotop} \ar[rr]^-{\Hom_A(B, - \otimes_R (K/R))} & &  \ind U \ar[d]^{\soc} \\
    & &\ind (\soc U).
  } \end{align*} 
 \end{enumerate}
\end{lemma}

\begin{proof} (a) Thanks to Lemma \ref{generalitiescotop} (c), $\cotop I \cong D_0 \top D_1 I$, so we only have to show that the $A^{\op}$-module $\Hom(B, \top D_1 I)$ is $0$ or simple. Suppose that $\Hom(B, \top D_1 I)$ is not $0$ or simple. We have two distinct maximal submodules $X_1, X_2 \subset D_1 I$ such that $S_1 := (D_1 I)/ X_1$ and $S_2 := (D_1 I)/X_2$ are simple $B^{\op}$-modules. By applying $\Hom_R(-,R)$ on the short exact sequence $0 \to X_1 \to D_1 I \to S_1 \to 0$, we get the short exact sequence $$0 \to I \xrightarrow{\iota_1} D_1 X_1 \to D_0 S_1 \to 0$$ and therefore $e \iota_1: eI \to e (D_1 X_1) $ is an isomorphism and $D_1 X_1 \in \CM_e A$. In the same way, $e \iota_2: eI \to  e (D_1 X_2)$ is an isomorphism and $D_1 X_2 \in \CM_e A$. We also get a non-split short exact sequence $0 \to Y \to X_1 \oplus X_2 \to D_1 I \to 0$. Applying $D_1$ to it, we get a short exact sequence $0 \to I \to D_1(X_1 \oplus X_2) \to D_1 Y \to 0$. Multiplying by $e$, we get the short exact sequence
 $$0 \to eI \xrightarrow{\begin{sbmatrix} e \iota_1 & e \iota_2 \end{sbmatrix}} e(D_1 X_1) \oplus e(D_1 X_1) \to e(D_1 Y) \to 0$$
which splits as $e \iota_1$ and $e \iota_2$ are isomorphisms. Thus $0 \to I \to D_1(X_1 \oplus X_2) \to D_1 Y \to 0$ is a non-split short exact sequence in $\CM_e A$. It is a contradiction as $\Ext^1_A(\CM_e A, I) = 0$.

 (b) Thanks to Lemma \ref{generalitiescotop} (e), a simple $B$-module $S$ is a direct summand of $\Bcotop I$ if and only if $\Ext^1_A(S, I) \neq 0$. Thus $\Bcotop I = 0$ if and only if $\Ext^1_A(S, I) = 0$ for any simple $B$-module $S$ if and only if $\Ext^1_A(\mod B, I) = 0$. 

 (c) Thanks to Proposition \ref{pair1}, $\soc Y \in \mod B$. Consider the sequence $0 \to I \to p^{-1}(\soc Y) \to \soc Y \to 0$. Thanks to Lemma \ref{generalitiescotop} (d), we have $\soc Y \hookrightarrow \Bcotop I$. 

 We will prove that for each direct summand $I'$ of $I$, $\Bcorad I'$ ($\subset X \otimes_R K$) is included in $X$. Consider the short exact sequence $0 \to I' \to X \to Y' \to 0$ induced by the inclusion $I' \subset I$.  
 As $i$ is radical, this short exact sequence does not split and we get $Y' \notin \CM_e A$ and $\soc Y' \neq 0$. Pulling back $0 \to I' \to X \to Y' \to 0$ along $\soc Y' \subset Y'$, we get a short exact sequence $0 \to I' \to X' \to \soc Y' \to 0$ with $X' \subset X$ so $X' \in \CM_e A$. Using (a) and Lemma \ref{generalitiescotop} (d), we obtain $\soc Y' \cong \Bcotop I'$ and therefore $X' = \Bcorad I' \subset X$. Finally $\Bcorad I \subset X$ and therefore $\Bcotop I \hookrightarrow Y$. As $\Bcotop I$ is semisimple, $\Bcotop I \hookrightarrow \soc Y$. So $\Bcotop I \cong \soc Y$. 

 (d) First of all, thanks to (a) and Lemma \ref{U is universal} (b), $\Bcotop$ induces a surjection from $\OO$ to $\ind (\soc U)$. Let us prove that it is injective. Suppose that $P, P' \in \OO$ satisfy $S := \Bcotop P = \Bcotop P'$ and consider the short exact sequences
 $$0 \to P \xrightarrow{f} \Bcorad P \xrightarrow{g} S \to 0 \quad \text{and} \quad 0 \to P' \xrightarrow{f'} \Bcorad P' \xrightarrow{g'} S \to 0.$$
 Multiplying them by $e$, we get $\Bcorad P, \Bcorad P' \in \CM_e A$. So, applying $\Hom_A(\Bcorad P, -)$ to the second short exact sequence, we get a morphism $u: \Bcorad P \to \Bcorad P'$ such that $g = ug'$. Symmetrically, we get $u': \Bcorad P' \to \Bcorad P$ such that $g' = u'g$. So $g = uu'g$ and, as $g$ is right minimal, $uu'$ is an isomorphism. Similarly, $u'u$ is an isomorphism so $\Bcorad P \cong \Bcorad P'$ and $P \cong P'$. We proved that $\Bcotop$ is injective on $\OO$. 

 The well definiteness of $\Hom_A(B, - \otimes_R (K/R)) : \OO \to \ind U$ is a direct consequence of the definition of $U$. The commutativity of the diagram is immediate by Lemma \ref{generalitiescotop} (a). As $U$ is injective, $\soc: \ind U \to \ind(\soc U)$ is bijective.
\end{proof}

The following proposition is used to categorify cluster algebras in Section \ref{parflag}.
\begin{proposition} \label{descexseq}
 If (E1) is satisfied, then the following assertions hold.
 \begin{enumerate}[\rm (a)]
  \item If $X \in \CM_e A$ does not have non-zero direct summands in $\add Ae$, then $TX \in \add \OO$. Moreover, $\Bcorad TX \subset X$ and $\Bcotop TX \cong \soc FX$.
  \item Let $0 \rightarrow X \rightarrow Y \rightarrow Z \rightarrow 0$ be a short exact sequence with $X, Z \in \CM_e A$ without non-zero direct summand in $\add Ae$. Then the maximal direct summand $Y_1$ of $Y$ in $\add Ae$ is the module satisfying $Y_1 \in \add \OO$ and $\soc FX \oplus \soc FZ \cong \soc FY \oplus \Bcotop Y_1$.
 \end{enumerate}
\end{proposition} 

\begin{proof}
 (a) Since $TX \to X$ is radical, the result follows from Lemma \ref{Bcotopsimple} (c).

 (b) Decompose $Y = Y_0 \oplus Y_1$. Recall that $T = Ae \otimes_{e A e} e -$ is exact on $\mod_e A$. As $TX$ is projective, we get 
 $$T Y_0  \oplus Y_1 = T Y \cong TX \oplus TZ \in \add \OO$$
 by (a). Again by (a), we get
 \begin{align*}\soc F X \oplus \soc F Z &\cong \Bcotop TX \oplus \Bcotop TZ \cong \Bcotop T Y_0 \oplus \Bcotop Y_1 \\ &\cong \soc F Y_0 \oplus \Bcotop Y_1 \cong \soc F Y \oplus \Bcotop Y_1. \qedhere \end{align*}
\end{proof}

\subsection{Proof of Theorem \ref{mainA}} 

(a) Since $Ae \in \add \Hom_R(gA,R)$, the conditions (E1) and (E2) are satisfied by Lemma \ref{puprem}.
By Theorem \ref{mainB} (c), we have an equivalence of exact categories
$B\otimes_A-:(\CM_eA)/[Ae]\cong\Sub U$ and $U$ is an injective $B$-module.
Thanks to Lemma \ref{U is universal} (b), we have $\soc U\cong\Bcotop Ae\cong \Hom_A(B, S_g)$. Thus $U\cong Q_g$.

(b) For $M \in \Sub Q_g$, let us consider a projective cover of $D_0 M$ in $\mod A^{\op}$:
 $$0 \rightarrow \Omega_A D_0 M \rightarrow P \rightarrow D_0 M \rightarrow 0.$$
 We have $P \in \add gA$. Applying $\Hom_R (-, R)$, we get the short exact sequence
 $$0 \rightarrow D_1 P \rightarrow D_1 \Omega_A D_0 M \rightarrow M \rightarrow 0.$$
 We have $D_1 P \in \add Ae$ so $D_1 \Omega_A D_0 M \in \CM_e A$ and $F(D_1 \Omega_A D_0 M) \cong M$ thanks to this sequence.

(c) Let us assume first that $A e$, $A f$ and $A g$ are basic. In particular $Ae \cong D_1(g A)$, $Af \cong D_1(e A)$ as $A$-modules and $eAe \cong D_1(eAe)$ as left ($eAe$)-modules. We have $e A f \cong e D_1(eA) = D_1(eAe) \cong e A e$ as left ($eAe$)-modules. So
 $Af \in \CM_e A$ and $T(Af) = Ae \otimes_{eAe} eAf \cong Ae$. Moreover, using the short exact sequence $0 \to T(Af) \to Af \to F(Af) \to 0$, we get
  $$\soc Bf = \soc F(Af) \subset \Bcotop T(Af) \cong \Bcotop Ae \cong \top B g$$
  so $Bf \subset D_0(g B)$. Dually, we get an inclusion $g B \subset D_0(Bf)$ by exchanging the role of $f$ and $g$. By comparing lengths over $R$ of $g B$ and $B f$, we deduce that $B f \cong D_0(g B) = Q_g$. 
 
 If $A e$, $A f$ or $A g$ are not basic, we take basic parts $e'$, $f'$ and $g'$ of $e$, $f$ and $g$ and we get $B f' \cong Q_{g'}$. Thus $\add Bf = \add Bf' = \add Q_{g'} = \add Q_g$.

(d) Since $A\in\CM_eA$, we have $B=FA\in\Sub Q_g$. Thus $\Sub Q_g=\Sub B$ holds by (c).

(e) All assumptions in Theorem \ref{frob} are satisfied. Moreover, since $A \in \CC$, the projective objects in $\EE=\mod_eA$ and $\CC=\CM_eA$ are projective $A$-modules, and the equivalent conditions of Theorem \ref{frob} (a) are satisfied. Thus applying Theorem \ref{frob} (d) (i) $\Leftrightarrow$ (iii), $B$ is Iwanaga-Gorenstein of dimension at most one if and only if $\CM_e A$ is Frobenius. As $A$ and $D_1 A$ are in $\CM_e A$, we get that $\CM A$ is Frobenius if and only if $\add A = \add D_1 A$ if and only if $\CM_e A$ is Frobenius, and the result follows. 

(f) In this case, $(\CM_e A)/[Ae] \cong \Sub B$ is an equivalence of Frobenius categories. Thus, since
  $\underline{\CM}_e A$ coincides with the stable category of $(\CM_e A)/[Ae]$, so $\underline{\CM}_e A \cong \underline{\Sub} B$ is a triangle equivalence.
\qed

\subsection{Proof of Theorem \ref{thmchangidem2}} \label{proofofchangidem}

By construction, we have an exact sequence
\begin{equation}\label{PM sequence}
0\rightarrow P_W\rightarrow W\rightarrow B\rightarrow 0
\end{equation}
with $W = Ae \oplus \widetilde{B} \in \CM A$ and $P_W = Ae \oplus P \in\add Ae$. Clearly we have $W\in\CM_{e}^B A$ and $P_W = Ae \otimes_{eAe} eW$. We set $A' := \End_A(W)$ and we identify $e$ with the idempotent of $A'$ which is the projection on the summand $Ae$ of $W$. We shall prove (a) in Proposition \ref{ApeB}, (b) in Proposition \ref{E1E2} and (d) in Proposition \ref{HT}. Then all hypotheses of Theorem \ref{mainB} are satisfied and the assertion (c) follows. Finally, (e) is an easy consequence of Proposition \ref{HT}.

\begin{lemma}\label{basic of M}
\begin{enumerate}[\rm (a)]
\item We have $We=Ae$, $W(1-e)=\widetilde{B}$ as $A$-module.
\item We have $W e A' = P_W$. Thus $P_W$ and $B$ have a structure of $A'^{\op}$-modules such that \eqref{PM sequence} is an exact sequence of $(A,A')$-bimodules.
\item We have $W/WeA'\cong B$ as $(A,A')$-bimodules.
\item We have $eA'=eW$. This is a projective $(eAe)$-module and a projective $A'{}^{\op}$-module.
\item We have $B \otimes_A W \cong B$ as $(B, A')$-bimodules.
\end{enumerate}
\end{lemma}

\begin{proof}
(a) Clear from definition.

(b) Since $We=Ae$, we have $WeA'=\sum_{f\in\End_A(W)}f(Ae)=Ae \otimes_{eAe} eW = P_W$.

(c) It is a clear consequence of (b).

(d) We have $eA'=\Hom_A(Ae,W)=eW$. Clearly $eA'$ is a projective $A'{}^{\op}$-module.
Moreover $eW=eP_W$ is a projective $(eAe)$-module since $P_W\in\add Ae$.

(e) Applying $B \otimes_A -$ to the short exact sequence \eqref{PM sequence}, we get the following exact sequence of $(B, A')$-bimodules:
 $$B \otimes_A P_W \to B \otimes_A W \to B \otimes_A B \to 0.$$
 Since $B \otimes_A P_W \in \add(B \otimes_A A e) = \add(Be) = \{0\}$ and $B \otimes_A B \cong B$, we get the result.
\end{proof}

\begin{proposition} \label{ApeB}
We have an isomorphism $A'/(e)\cong B$ of $R$-algebras.
\end{proposition}

\begin{proof}
Applying $\Hom_A(W,-)$ to \eqref{PM sequence}, we have an exact sequence
\[0\rightarrow \Hom_A(W,P_W)\rightarrow A'\rightarrow \Hom_A(W,B)\rightarrow \Ext^1_A(W,P_W),\]
where $\Ext^1_A(W,P_W)=0$ by $P_W\in\add Ae$, $W \in \CM_e^B A$ and our assumption $\Ext^1_A(\CM_e^B, Ae) = 0$.
Since $\Hom_A(Ae,B)=0$, applying $\Hom_A(-, B)$ to \eqref{PM sequence}, we have $\Hom_A(W,B)=\End_A(B)=B$ and $(e)=\Hom_A(W,P_W)$. Thus
$A'/(e)=A'/\Hom_A(W,P_W)=\Hom_A(W,B)=B$.
\end{proof}

In particular, we can regard $\mod B$ as full subcategory of both $\mod A'$ and $\mod A$.
Now we consider the adjoint pair $(G,H)$ given by 
\[H:=\Hom_A(W,-):\mod A\rightarrow \mod A'\mbox{ and }G:=W\otimes_{A'}-:\mod A'\rightarrow \mod A.\]

The main result about these functors is:

\begin{proposition}\label{HT}
 The adjoint pair $(G,H)$ gives quasi-inverse equivalences of exact categories between $\mod_{e}^{B}A$ and $\mod_{e}A'$, which restrict to quasi-inverse equivalences of exact categories between $\CM_{e}^{B}A$ and $\CM_{e}A'$.
\end{proposition}

The first step of the proof consists of the following lemma.

\begin{lemma}\label{HTB}
\begin{enumerate}[\rm (a)]
\item $H$ and $G$ give quasi-inverse equivalences between $\add Ae$ and $\add A'e$.
\item We have commutative diagrams
\[\xymatrix{
\mod A\ar[rr]^H&&\mod A'&&\mod A&&\mod A'\ar[ll]_{G}\\
&\mod B\ar[ul]\ar[ru]&&&&\mod B\ar[ul]\ar[ru]
}\]
\end{enumerate}
\end{lemma}

\begin{proof}
(a) This is clear: $H(Ae)=\Hom_A(W,Ae)=A'e$ and $G(A'e)\cong We=Ae$ by Lemma \ref{basic of M}.

(b) Fix $X\in\mod B$. Applying $\Hom_A(-,X)$ to \eqref{PM sequence}, we have an exact sequence
\[0\rightarrow \Hom_A(B,X)\rightarrow HX\rightarrow \Hom_A(P_W,X),\]
where $\Hom_A(P_W,X)=0$ by $P_W\in\add Ae$ and $X\in\mod B$.
Thus we have
\[HX \cong \Hom_A(B,X)\cong X.\]
On the other hand, we have
\[G(X)=W\otimes_{A'}X=W\otimes_{A'}(B\otimes_{A'}X)=(W/WeA')\otimes_{A'}X\stackrel{{\rm Lem. \ref{basic of M}}}{=}B\otimes_{A'}X=X.\qedhere\]
\end{proof}

\begin{lemma}\label{tor=0}
\begin{enumerate}[\rm (a)]
\item We have $\Tor^{A'}_1(Y,X)=\Tor^B_1(Y\otimes_{A'}B,X)$ for any $X\in\mod B$ and $Y\in\CM A'{}^{\op}$.
\item We have $\Tor^{A'}_1(W,X)=0$ for any $X\in\mod_eA'$.
\end{enumerate}
\end{lemma}

\begin{proof}
 For $Y\in\CM A'{}^{\op}$, take an exact sequence
\begin{equation}\label{A' resolution}
0\rightarrow \Omega Y\xrightarrow{i} P\rightarrow Y\rightarrow 0
\end{equation}
of $A'{}^{\op}$-modules with $P\in\proj A'{}^{\op}$.
We will show that
\begin{equation}\label{B resolution}
0\rightarrow \Omega Y\otimes_{A'}B\xrightarrow{i\otimes 1_B} P\otimes_{A'}B\rightarrow Y\otimes_{A'}B\rightarrow 0
\end{equation}
is exact. Since $eA'\in\proj(eA'e)$, we have an exact sequence $0\rightarrow A'e\otimes_{eA'e}eA'\stackrel{j}{\rightarrow}A'\rightarrow B\rightarrow 0$ of $(A',A')$-bimodules.
Applying $Y\otimes_{A'}-$, we have an exact sequence
\[0\rightarrow K\rightarrow Ye\otimes_{eA'e}eA'\xrightarrow{1_Y\otimes j}Y\rightarrow Y\otimes_{A'}B\rightarrow 0.\]
Since $(1_Y\otimes j)e:(Ye\otimes_{eA'e}eA')e\rightarrow Ye$ is an isomorphism, we have $Ke=0$. Thus $K$ is a $B^{\op}$-module. Thus $K=0$ holds since $Ye\otimes_{eA'e}eA'\in\CM A'{}^{\op}$ holds again by $eA'\in\proj(eA'e)$.

Applying the same argument to $P\in\CM A'{}^{\op}$ and $\Omega Y\in\CM A'{}^{\op}$, we have the following commutative diagram of exact sequences:
\[\xymatrix{
&&0\ar[d]\\
0\ar[r]&\Omega Ye\otimes_{eA'e}eA\ar[r]\ar[d]&\Omega Y\ar[r]\ar[d]&\Omega Y\otimes_{A'}B\ar[r]\ar[d]^{i\otimes 1_B}&0\\
0\ar[r]&Pe\otimes_{eA'e}eA\ar[r]\ar[d]&P\ar[r]\ar[d]&P\otimes_{A'}B\ar[r]\ar[d]&0\\
0\ar[r]&Ye\otimes_{eA'e}eA\ar[r]\ar[d]&Y\ar[r]\ar[d]&Y\otimes_{A'}B\ar[r]\ar[d]&0\\
&0&0&0
}\]
A diagram chase shows that $i\otimes 1_B$ is injective. Thus \eqref{B resolution} is exact.

(a) For $X \in \mod B$, applying $-\otimes_{A'}X$ to \eqref{A' resolution} and $-\otimes_BX$ to \eqref{B resolution} and comparing them, we have a commutative diagram of exact sequences:
\[\xymatrix{
0\ar[r]&\Tor^{A'}_1(Y,X)\ar[r]&\Omega Y\otimes_{A'}X\ar[r]\ar@{=}[d]&P\otimes_{A'}X\ar@{=}[d]\\
0\ar[r]&\Tor^B_1(Y\otimes_{A'}B,X)\ar[r]&(\Omega Y\otimes_{A'}B)\otimes_BX\ar[r]&(P\otimes_{A'}B)\otimes_BX
}\]
Thus the assertion follows.

(b) First, we assume $X\in\mod B$. Since $W\in\CM A'{}^{\op}$, by (a) and Lemma \ref{basic of M} (c), we have
\[\Tor^{A'}_1(W,X)\cong\Tor^B_1(W\otimes_{A'}B,X)\cong\Tor^B_1((W/W e A') \otimes_{A'}B,X)\cong\Tor^B_1(B,X)=0.\]
Now we assume $X\in\mod_eA'$. Then there exists an exact sequence $0\rightarrow P\rightarrow X\rightarrow Y\rightarrow 0$ with $P\in\add A'e$ and $Y\in\mod B$.
Applying $W\otimes_{A'}-$, we have an exact sequence
\[0=\Tor^{A'}_1(W,P)\rightarrow \Tor^{A'}_1(W,X)\rightarrow \Tor^{A'}_1(W,Y)=0.\]
Thus the assertion follows.
\end{proof}

\begin{proof}[Proof of Proposition \ref{HT}]
(i) First we show $H(\mod_{e}A)\subset\mod_{e}A'$.

For $X\in\mod_{e}A$, we have
\[eH(X)=\Hom_A(We,X)\stackrel{}{=}\Hom_A(Ae,X)=eX\in\proj(eAe)=\proj(eA'e).\]

(ii) Next we show $G(\mod_{e}A')\in\mod_{e}^{B}A$.

For $X\in\mod_{e}A'$, take an exact sequence
\begin{equation}\label{sequence1}
0\rightarrow P\rightarrow X\rightarrow Y\rightarrow 0
\end{equation}
with $P\in\add A'e$ and $Y\in\mod B$. Applying $G$, we have a short exact sequence
\begin{equation}\label{sequence2}
0\rightarrow GP\rightarrow GX\rightarrow GY\rightarrow 0.
\end{equation}
by Lemma \ref{tor=0}. Since $GP\in\add Ae$ and $GY=Y\in\mod B$ thanks to Lemma \ref{HTB}, we have $GX\in\mod_{e}^{B}A'$.

(iii) We show $HG\cong{\rm id}_{\mod_{e}A'}$ and $GH\cong{\rm id}_{\mod_{e}^{B}A}$.

Applying $H$ to \eqref{sequence2} and comparing with \eqref{sequence1}, we have a commutative diagram of exact sequences,
\[\xymatrix{
0\ar[r]&P\ar[r]\ar[d]&X\ar[r]\ar[d]&Y\ar[r]\ar[d]&0\\
0\ar[r]&HGP\ar[r]&HGX\ar[r]&HGY
}\]
where vertical arrows are of the form $x \mapsto (w \mapsto w \otimes x)$. Since the left and the right vertical maps are isomorphisms, so is the middle one.

By a similar argument, one can show $GH\cong{\rm id}_{\mod_{e}^{B}A}$.

(iv) We show that $H$ and $G$ are exact functors. The functor $G$ is exact thanks to Lemma \ref{tor=0}. Applying $\Hom_A(-, \mod B)$ to \eqref{PM sequence}, we get an exact sequence
$$0 = \Ext^1_A(B, \mod B) \rightarrow \Ext^1_A(W, \mod B) \rightarrow \Ext^1_A(P_W, \mod B) = 0$$
so $\Ext^1_A(W, \mod B) = 0$. As we also have $\Ext^1_A(W, Ae) = 0$, we get $\Ext^1_A(W, \mod_e^B A) = 0$. So $H$ is exact.

(v) We now show that the equivalences restrict to $\CM_e^B A \cong \CM_e A'$. 

Clearly $H(\CM_{e}^{B}A)\subset\CM_{e}A'$ holds.
It is enough to show that, if $X\in\mod_{e}^{B}A$ satisfies $HX\in\CM_{e}A'$, then $X\in\CM A$.
Let $Y$ be a finite length submodule of $X$.
Then the inclusion $Y\subset X$ gives an injection $HY\subset HX$. Since $HY$ has finite length and $HX\in\CM_{e}A'$, we have $HY=0$.

Let $0\rightarrow P\stackrel{i}{\rightarrow}X\rightarrow Z\rightarrow 0$ be an exact sequence with $P\in\add Ae$ and $Z\in\mod B$. 
Since $Y \cap P = 0$, we have that $Y$ is a submodule of $Z$.
In particular $Y\in\mod B$. Since $HY=0$, we have $Y=0$ by (iii). Thus $X\in\CM A$.
\end{proof}

\begin{proposition} \label{E1E2}
We have (E1) $\Ext^1_{A'}(\CM_{e}A',A'e)=0$ and (E2)\pup* $\Ext^2_{A'}(\mod_{e}A',A'e)=0$.
\end{proposition}

\begin{proof}
(E1) Let $0\rightarrow A'e\rightarrow X\rightarrow Y\rightarrow 0$ be an exact sequence with $Y\in\CM_{e}A'$.
Applying $G$ and using Lemma \ref{tor=0}, we have an exact sequence
$0\rightarrow G(A'e)\rightarrow GX\rightarrow GY\rightarrow 0$.
It splits since $\Ext^1_A(GY,Ae)=0$ by our assumption.
Since $G:\CM_{e}A'\rightarrow \CM_{e}^{B}A$ is an equivalence, the original sequence splits. Thus the assertion follows.

(E2)\pup* Since we have $A'\in\CM_eA'$ by Lemma \ref{basic of M} (d), syzygies of modules in $\mod_eA'$ belongs to $\CM_eA'$. 
Thus the assertion follows from (E1).
\end{proof}

We finish this subsection by proving Lemma \ref{lemmagorchangeorders}.

\begin{proof}[Proof of Lemma \ref{lemmagorchangeorders}]
  As $Ae \cong D_1 (g A)$ is injective in $\CM A$ and $\CM_e^B A \subset \CM A$, we get (C3).

  To prove the second part of the statement, let us prove that if (C1) holds, then for a finite length $A$-module $M$, we have $M \in \Sub(Ae \otimes_R (K/R))$ if and only if the socle of $M$ is supported by $g$. As $Ae$ is injective in $\CM A$ and syzygies of all modules are Cohen-Macaulay, we have $\Ext^2_A(\mod A, Ae) = 0$. By Lemma \ref{injfl}, we have $\Ext^1_A(\fl A, Ae \otimes_R K) = 0$. So applying $\Hom_A(\fl A, -)$ to the short exact sequence $$0 \to Ae \to Ae \otimes_R K \to Ae \otimes_R (K/R) \to 0,$$ we get $\Ext^1_A(\fl A, Ae \otimes_R (K/R)) = 0$. Moreover, by Lemma \ref{generalitiescotop} (a), we get that $\soc(Ae \otimes_R (K/R))$ is the semisimple module corresponding to $g$. 

  It is immediate that if $M \in \Sub(Ae \otimes_R (K/R))$ then $\soc M \in \add \soc (Ae \otimes_R (K/R))$ so the first implication is satisfied. Conversely, if $\soc M$ is supported by $g$, there exists a injection $\soc M \hookrightarrow (Ae \otimes_R (K/R))^{\oplus \ell}$. Then, applying $\Hom_A(-, (Ae \otimes_R (K/R))^{\oplus \ell})$ to the short exact sequence $$0 \to \soc M \to M \to M/\soc M \to 0,$$ there is an injection $M  \hookrightarrow (Ae \otimes_R (K/R))^{\oplus \ell}$. So we proved the converse implication.
 \end{proof}

\section{Cluster algebra structure on coordinate rings of partial flag varieties} \label{parflag}

The aim of this section is to apply results in previous sections to categorify
the cluster algebra structure of the multi-homogeneous coordinate rings $\Cf[\FF]$ of the partial flag variety
$\FF=\FF(\Delta,J)$ corresponding to a Dynkin diagram $\Delta$ and a set $J$ of vertices of $\Delta$ by using the category of Cohen-Macaulay modules.
To be more precise, recall that Geiss-Leclerc-Schr\"oer introduced in \cite{GeLeSc08} a cluster algebra $\tilde \AA \subset \Cf[\FF]$. They proved that $\tilde \AA = \Cf[\FF]$ in type $A_n$. In general, they conjecture that $\tilde \AA[\Sigma_J^{-1}] = \Cf[\FF][\Sigma_J^{-1}]$ where $\Sigma_J$ is the set of principal generalized minors corresponding to non-minuscule weights (see Definition \ref{genmin} of principal generalized minors), and they prove the conjecture in type $D_4$.

The main result of this section (Theorem \ref{categG}) consists in completing Gei\ss-Leclerc-Schr\"oer's partial categorification of $\tilde \AA$. Their categorification (Theorem \ref{categGLS})
uses the preprojective algebra $\Pi=\Pi(\Delta)$ over $\Cf$ and the full subcategory $\Sub Q_J$ of $\mod\Pi$, where $Q_J$ is the direct sum of indecomposable injective $\Pi$-modules corresponding to vertices in $J$. Recall that a Frobenius category $\EE$ is said to be \emph{stably $2$-Calabi-Yau} if there is a bifunctorial isomorphism $\Ext^1_\EE(X, Y) \cong \Ext^1_\EE(Y,X)$ and that $\Sub Q_J$ is stably $2$-Calabi-Yau. Moreover, an object $X$ in $\EE$ is called \emph{rigid} if $\Ext^1_\EE(X, X) = 0$ and it is called \emph{cluster tilting} if $\add X = \{Y \in \EE \mid \Ext^1_\EE(Y, X) = 0\}$. 

\subsection{The categorification of Geiss-Leclerc-Schr\"oer} \label{catGLS}
 In this section, we recall briefly the results of \cite{GeLeSc08} concerning the categorification of cluster algebra structures on multi-homogeneous coordinate rings of partial flag varieties. We start by fixing a simple simply connected complex algebraic group $G$ with Dynkin diagram $\Delta$. We fix a maximal torus $H \subset G$ and two opposite Borel subgroups $B, B^- \subset G$ satisfying $B \cap B^- = H$ (for more details about Lie theoretical background, see \cite{Bo91, LaGo01}). For a vertex $i$ of $\Delta$, we fix $$x_i(t) := \exp(t e_i) \quad \text{and} \quad y_i(t) := \exp(t f_i)$$ the one parameter subgroups of $B$ and $B^-$ corresponding to the Chevaley generators $e_i$ and $f_i$ of the Lie algebra of $G$. Following notations of \cite{GeLeSc08}, we define $K$ to be the complement of $J$. The \emph{parabolic subgroup} $B_K$ (respectively, \emph{opposite parabolic subgroup} $B^-_K$) of $G$ is the subgroup generated by $B$ (respectively, $B^-$) and $y_i$ (respectively, $x_i$) for $i \in K$. The partial flag variety $\FF$ can be realized as $\FF = B^-_K \backslash G$. Let $N_K$ be the unipotent radical of $B_K$, that is the subgroup of unipotent elements of the maximal solvable normal subgroup of $B_K$. Then, it is a classical result that $N_K \subset G$ induces an embedding $N_K \subset \FF$ as a dense affine open subset.
 \begin{example} \label{exflags}
  If $\Delta = A_4$ and $J = \{1,3\}$, we have $K = \{2, 4\}$, $G = \SL_5(\Cf)$ and 
  $$B^-_K = \begin{bmatrix}
             * & 0 & 0 & 0 & 0 \\
             * & * & * & 0 & 0 \\
             * & * & * & 0 & 0 \\
             * & * & * & * & * \\
             * & * & * & * & *
            \end{bmatrix} \subset G \quad \text{and} \quad N_K = \begin{bmatrix}
             1 & * & * & * & * \\
             0 & 1 & 0 & * & * \\
             0 & 0 & 1 & * & * \\
             0 & 0 & 0 & 1 & 0 \\
             0 & 0 & 0 & 0 & 1
            \end{bmatrix} $$
  and it is immediate that $B^-_K \backslash G$ parametrizes naturally flags of $\Cf^5$ of type $1, 3$.
 \end{example}
 
 Let $\ig = (i_1, i_2, \dots, i_\ell)$ be a sequence of vertices of $\Delta$, $\kg = (k_1, k_2, \dots, k_\ell)$ be a sequence of non-negative integers and $\tg = (t_1, t_2, \dots, t_\ell)$ be a sequence of variables. We denote 
 \begin{itemize}
  \item $\ig^\kg$ the sequence of indices obtained from $\ig$ by repeating $k_j$ times $i_j$;
  \item $\tg^\kg := t_1^{k_1} t_2^{k_2} \cdots t_\ell^{k_\ell}$;
  \item $\kg! := k_1! k_2! \cdots k_\ell!$;
  \item $x_\ig(\tg) := x_{i_1}(t_1) x_{i_2}(t_2) \cdots x_{i_\ell}(t_\ell)$.
 \end{itemize}
 We denote by $\Phi_{M, \ig}$ the variety of composition series of $M$ of type $\ig$, that is
  $$\Phi_{M, \ig} := \{0 = M_0 \subset M_1 \subset \cdots \subset M_\ell = M \,|\, \forall j, M_j / M_{j-1} \cong S_{i_j}\}$$
  realized within the appropriate product of Grassmannians. Finally $\chi$ is the Euler characteristic.

 In \cite{GeLeSc05}, using Lusztig's semicanonical basis \cite{Lu00}, Geiss-Leclerc-Schr\"oer define functions in the coordinate ring $\Cf[N] = \Cf[N_\emptyset]$ by the following result:
 \begin{theorem}[\cite{Lu00, GeLeSc05}]
  Let $M \in \mod \Pi$. There exists a unique function $\varphi_M$ in $\Cf[N]$ satisfying 
  $$\varphi_M(x_\ig(\tg)) = \sum_{\kg \in \N^\ell} \chi(\Phi_{M, \ig^\kg}) \frac{\tg^\kg}{\kg!}$$
  for any reduced word $\ig$ of an element of the Weyl group of type $\Delta$. 
 \end{theorem}
 In \cite{GeLeSc05}, they also prove that 
 \begin{itemize}
  \item $\varphi_{Y \oplus Z} = \varphi_Y \varphi_Z$ for any $Y, Z \in \mod \Pi$;
  \item if $Y$ and $Z$ are indecomposable such that $\dim \Ext^1_\Pi(Y, Z) = 1$ and $$0 \to Y \to U \to Z \to 0 \quad \text{and} \quad 0 \to Z \to U' \to Y \to 0$$ are two non-split short exact sequences, then $\varphi_Y \varphi_Z = \varphi_U + \varphi_{U'}$.
 \end{itemize}
 In other terms, $\varphi$ is a so-called \emph{cluster character}. 

 In \cite{GeLeSc08}, the authors prove that $\Sub Q_J$ categorifies via $\varphi$ and the canonical projection $\Cf[N] \twoheadrightarrow \Cf[N_K]$ a cluster algebra $\AA \subset \Cf[N_K]$. They prove in type $A_n$ and $D_4$ that $\AA = \Cf[N_K]$ and they conjecture it to be true in any case.

 Let us introduced generalized principal minors (see \cite{FoZe99}):
 \begin{definition} \label{genmin}
  For a vertex $i$ of $\Delta$, the corresponding \emph{principal generalized minor} is defined on $G$ as the unique function $\Delta_i$ satisfying
  $$\Delta_i(x^- x_0 x^+) = \Delta_i(x_0)$$
  for $x^- \in B^-$, $x_0 \in H$, $x^+ \in B$ and $\Delta_i|_H: H \to \Cf^*$ is the multiplicative character corresponding to the fundamental weight indexed by $i$.
 \end{definition}
 
 It is known that $\FF = B^-_K \backslash G$ is embedded in a product of projective spaces indexed by $J$ (in type $A_n$, a product of usual Grassmannians). Thus, we can define the \emph{multi-homogeneous coordinate ring} $\Cf[\FF]$, graded by $\N^J$. Each of the $\Delta_j$ is homogeneous of degree $(0, \dots, 0, 1, 0, \dots, 0)$ where $1$ is at position $j$ and $N_K$ is the open dense affine subset of $\FF$ defined by $N_K = \{x \in \FF \,|\, \forall j \in J, \Delta_j(x) \neq 0\}$ so there is a dehomogenization map $\Cf[\FF] \twoheadrightarrow \Cf[N_K]$ defined by mapping $\Delta_j$ to $1$. For any $f \in \Cf[N_K]$ there is a unique homogeneous $\tilde f \in \Cf[\FF]$ such that $\pi(\tilde f) = f$ and the multi-degree of $\tilde f$ is minimal for the order induced by fundamental weights \cite[Lemma 2.4]{GeLeSc08}. 

 \begin{example}
  We continue Example \ref{exflags}. It this case, $\Delta_1$ corresponds to the upper-left coefficient and $\Delta_3$ corresponds to the determinant of the upper-left ($3 \times 3$)-submatrix. Then $B_K^- \backslash G$ is a closed subset of $\Gr_1(\Cf^5) \times \Gr_3(\Cf^5)$, by mapping $M \in B^-_K$ to the subspaces generated by the first row on the one hand and the first three rows on the second hand. So, as usual, thanks to Pl\"ucker coordinates, we have $$\FF \subset \Gr_1(\Cf^5) \times \Gr_3(\Cf^5) \subset \P\left(\Cf^{\binom{5}{1}}\right) \times \P\left(\Cf^{\binom{5}{3}}\right).$$ Then, we have two affine subspaces $N_{\{1\}^{c}}$ of $\Gr_1(\Cf^5)$ and $N_{\{3\}^{c}}$ of $\Gr_3(\Cf^5)$ defined by the non-vanishing of the leftmost determinants, which are Pl\"ucker coordinates and correspond to $\Delta_1$ and $\Delta_3$ as functions over $G$. Moreover, $N_K = \left(N_{\{1\}^{c}} \times N_{\{3\}^{c}}\right) \cap \FF$.
 \end{example}

 In order to extend the cluster algebra $\AA \subset \Cf[N_K]$ to a cluster algebra $\tilde \AA \subset \Cf[\FF]$ by adding coefficients $\Delta_j$ corresponding to the multi-homogenization, Geiss-Leclerc-Schr\"oer prove the following theorem.
 \begin{theorem}[{\cite[10.1]{GeLeSc08}}] \label{mutGLS}
  If $Y, Z \in \Sub Q_J$, then $\tilde \varphi_{Y \oplus Z} = \tilde \varphi_Y \tilde \varphi_Z$. If $Y, Z \in \Sub Q_J$ satisfy $\dim \Ext^1_\Pi(Y, Z) = 1$, and
  $$0 \to Y \to U \to Z \to 0 \quad \text{and} \quad 0 \to Z \to U' \to Y \to 0$$
  are non-split short exact sequences then
  $$\tilde \varphi_Y \tilde \varphi_Z = \tilde \varphi_U \prod_{j \in J} \Delta_j^{\alpha_j} + \tilde \varphi_{U'} \prod_{j \in J} \Delta_j^{\beta_j}$$
  where
  \begin{align*} \alpha_j &= \max(0, \dim \Hom_\Pi(S_{j}, U') - \dim \Hom_\Pi(S_{j}, U)) \\ \text{and} \quad \beta_j &= \max(0, \dim \Hom_\Pi(S_{j}, U)- \dim \Hom_\Pi(S_{j}, U')).\end{align*}
 \end{theorem}

 To construct $\tilde \AA$ using Theorem \ref{mutGLS}, Geiss-Leclerc-Schr\"oer constructed an explicit cluster tilting object in $\Sub Q_J$ that they call \emph{initial}. A cluster tilting object in $\Sub Q_J$ is called \emph{reachable} if it is obtained from the initial one by successive mutations. An indecomposable rigid object is called \emph{reachable} if it is a direct summand of a reachable cluster tilting object. Their result can be stated as follows.
 \begin{theorem}[{\cite[Theorem 10.2]{GeLeSc08}}] \label{categGLS}
  \begin{enumerate}[\rm (a)]
   \item There is a cluster algebra $\tilde \AA \subset \Cf[\FF]$ such that:
    \begin{itemize}
     \item coefficients of $\tilde \AA$ are $\tilde c$ for each coefficient $c$ of $\AA$ and $\Delta_j$ for each $j \in J$;
     \item clusters of $\tilde \AA$ are $\{\tilde x_1, \tilde x_2, \dots, \tilde x_\ell\} \sqcup \{\Delta_j \,|\, j \in J\}$ for each cluster $\{x_1, x_2, \dots, x_\ell\}$ of $\AA$.
    \end{itemize}  
   \item There is a bijection $X \mapsto \tilde \varphi_X$ between
    \begin{itemize}
     \item isomorphism classes of reachable indecomposable rigid objects of $\Sub Q_J$;
     \item cluster variables and coefficients of $\tilde \AA$ except $\Delta_j$ for $j \in J$.
    \end{itemize}
   \item There is a bijection $\bigoplus_{k=1}^\ell T_k \mapsto \{\tilde \varphi_{T_1}, \tilde \varphi_{T_2}, \dots, \tilde \varphi_{T_\ell}\} \sqcup \{\Delta_j \,|\, j \in J\} $ between
    \begin{itemize}
     \item isomorphism classes of reachable basic cluster tilting objects of $\Sub Q_J$;
     \item clusters of $\tilde \AA$.
    \end{itemize}
   Moreover, it commutes with mutation of cluster tilting objects and mutation of clusters.  
  \end{enumerate} 
 \end{theorem}

\subsection{Categorification of the cluster algebra structure of $\Cf[\FF]$ using $\CM_e A$}
We keep the setting of the beginning of this section, and we fix $R := \Cf\llbracket t \rrbracket$. 
Our aim is to categorify $\Cf[\FF(\Delta, J)]$ by a category $\CM_e A$ where $A$ is an $R$-order and $e \in A$ is an idempotent. We denote by $g = g_J$ the idempotent of $\Pi$ corresponding to the set $J$. We also denote $I_J := \Hom_\Pi(\Pi/(g), \Pi)$ which is the biggest ideal of $\Pi$ satisfying $g I_J = 0$. We observe that
\begin{itemize}
 \item injective modules corresponding to $j \in J$ in $\mod \Pi$ and $\mod \Pi/I_J$ coincide;
 \item $\Pi/I_J \in \Sub Q_J \subset \mod \Pi/I_J \subset \mod \Pi$.
\end{itemize}

We define pairs $(A, e)$ permitting the categorification. 
\begin{definition} \label{dfmod}
A pair $(A, e)$ where $A$ is an $R$-order and $e \in A$ is an idempotent \emph{models} $(\Delta, J)$ if 
\begin{itemize}
 \item $B := A/(e) \cong \Pi(\Delta)/I_J$ as $\Cf$-algebras;
 \item (E1) $\Ext^1_A(\CM_e A, Ae) = 0$ and (E2) $\Ext^2_{\mod_e A}(\mod_e A, Ae) = 0$;
 \item $\Bcotop$ induces a bijection from $\ind Ae$ to $\ind(\soc Q_J)$.
\end{itemize}
\end{definition}

Using the last condition of Definition \ref{dfmod}, if $(A,e)$ models $(\Delta, J)$, we can decompose $e$ as sum of primitive orthogonal idempotents $e = \sum_{j \in J} e_j$ in such a way that for every $j \in J$, \begin{equation} \Bcotop A e_j \cong \soc Q_j. \label{corcotop} \end{equation}

In this context, we have the following equivalence of category:
\begin{lemma} \label{eq6}
 If $(A,e)$ models $(\Delta, J)$, $B \otimes_A -$ restrict to an exact bijective functor $F: \CM_e A \to \Sub Q_J$ which induces an equivalence of exact categories
 $(\CM_e A) / [Ae] \to \Sub Q_J$.
\end{lemma}

\begin{proof}
 Thanks to Theorem \ref{mainB} (c) and (d), $F := B \otimes_A - : \CM_e A \to \Sub U$ induces an equivalence of exact categories $(\CM_e A)/[Ae] \to \Sub U$ for some injective $B$-module $U$, so $F$ is exact bijective. By Lemma \ref{Bcotopsimple} (d), we have $U \cong Q_J$, hence the statement holds.  
\end{proof}

We start by proving the following proposition by applying the method of change of orders given in Theorem \ref{thmchangidem2}.

\begin{proposition} \label{changeJ}
 Assume that $(A,e)$ models $(\Delta, J)$. Then, for any subset $J'$ of $J$, there exists an order $A'$, explicitly constructed from $A$, and an idempotent $e'$ of $A'$ such that $(A', e')$ models $(\Delta, J')$.
\end{proposition}

\begin{proof}
 First of all, using indices of \eqref{corcotop}, let $e' = \sum_{j \in J'} e_j$. Denote $B := \Pi/I_J$ and $B' := \Pi/ I_{J'}$. Then $B'$ is a quotient of $A/(e')$. Let us check that $(A, e')$ and $B'$ satisfy the hypotheses of Theorem \ref{thmchangidem2}. First of all, (C1) is clear. By Lemma \ref{eq6}, $Q_{J'} \cong FX$ for some $X \in \CM_e A$. Moreover, according to Proposition \ref{descexseq} (a), $\Bcotop TX \cong \soc Q_{J'}$ so $TX \cong Ae'$. Therefore, thanks to Lemma \ref{LemmaFX}, we get
 $$B' \in \Sub Q_{J'} \subset \Sub (Ae' \otimes_R (K/R)),$$
 hence (C2) is satisfied.
 It is immediate that $\CM_{e'}^{B'} A \subset \CM_e A$ so, thanks to $(E1)$, we get (C3) $\Ext^1_A (\CM_{e'}^{B'} A, Ae') = 0$. 

 We apply Theorem \ref{thmchangidem2} to the pair $(A, e')$ to get an explicit order $A'$. Let us show that $(A', e')$ models $(\Delta, J')$. We have $B' \cong A'/(e')$ by Theorem \ref{thmchangidem2} (a). Moreover, $(A', e')$ satisfies (E1) and (E2)\pup* by Theorem \ref{thmchangidem2} (b), so it also satisfies (E2). It remains to check for $j \in J'$ that $\Bpcotop (A' e_j) \cong S_j$. Thanks to Proposition \ref{HT} and Lemma \ref{HTB}, applying $H$ to $0 \to A e_j \to \Bcorad (A e_j) \to S_j \to 0$ gives a short exact sequence
 $0 \rightarrow A' e_j \rightarrow H ( \Bcorad (A e_j) ) \rightarrow S_j \rightarrow 0$
 which does not split. Moreover, $H (\Bcorad (A e_j) ) \in \CM_{e'} A'$ so $S_j$ is a summand of $\Bpcotop(A' e_j)$. So, thanks to Lemma \ref{Bcotopsimple}, $\Bpcotop(A' e_j) \cong S_j$. 
\end{proof}

As a consequence, we obtain the following important result of this paper:
\begin{theorem} \label{modelAD}
 For any Dynkin diagram $\Delta$ and any set $J$ of vertices of $\Delta$ there exists a pair $(A,e)$ which models $(\Delta, J)$.
\end{theorem}

\begin{proof}
 As $\Pi$ is selfinjective, thanks to Corollary \ref{thselfinj}, there exist an order $A$ and an idempotent $e$ of $A$ such that $A/(e) \cong \Pi$ as $\Cf$-algebras and $D_1 (Ae) \cong (1-e) A$ as right $A$-modules. So it is immediate that $(A, e)$ models $(\Delta, \Delta_0)$ where $\Delta_0$ is the set of vertices of $\Delta$.
 Then, Proposition \ref{changeJ} allows us to conclude immediately.
\end{proof}

Notice that the pair $(A,e)$ in Theorem \ref{modelAD} is not unique. We will construct in \cite{DeIy-2} other possibilities than the one considered in this paper.

We now fix a pair $(\Delta, J)$ and a pair $(A, e)$ modelling it. We will prove that $\CM_e A$ categorifies the cluster algebra structure of $\tilde \AA$.
From now on, we consider $F: \CM_e A \rightarrow \Sub Q_J$ as in Lemma \ref{eq6}. Since the category $\Sub Q_J$ is stably $2$-Calabi-Yau, $\CM_e A$ is also stably $2$-Calabi-Yau. We now extend the character $\tilde \varphi$ to $\CM_e A$:

\begin{definition}
 For $Y \in \CM_e A$, we define $\psi_Y \in \tilde \AA$ as follows. If $Y$ does not have non-zero direct summands in $\add Ae$, then $\psi_Y := \tilde \varphi_{FY}$. For $j \in J$, $\psi_{A e_j} := \Delta_j$, and we extend the definition to $\CM_e A$ by the property $\psi_{Y \oplus Z} = \psi_Y \psi_Z$.
\end{definition}

The following main result of this subsection improves Theorem \ref{categGLS} of Geiss-Leclerc-Schr\"oer:

\begin{theorem} \label{categG}
  \begin{enumerate}[\rm (a)]
   \item $\psi$ induces a bijection between
    \begin{itemize}
     \item isomorphism classes of reachable indecomposable rigid objects of $\CM_e A$;
     \item cluster variables and coefficients of $\tilde \AA$.
    \end{itemize}
   \item $\psi$ induces a bijection between
    \begin{itemize}
     \item isomorphism classes of reachable basic cluster tilting objects of $\CM_e A$;
     \item clusters of $\tilde \AA$.
    \end{itemize}
   Moreover, it commutes with mutation of cluster tilting objects and mutation of clusters.  
  \end{enumerate} 
 \end{theorem}

 We start by proving that $\psi$ is a cluster character, extending Theorem \ref{mutGLS}:

\begin{proposition} \label{mutG} \begin{enumerate}[\rm (a)]
 \item If $Y, Z \in \CM_e A$, then $\psi_{Y \oplus Z} = \psi_Y \psi_Z$. 
 \item If $Y, Z \in \CM_e A$ are indecomposable and $\dim \Ext^1_A(Y,Z) = 1$ (or equivalently $\dim \Ext^1_A(Z,Y) = 1$), we have $\psi_Y \psi_Z = \psi_U + \psi_{U'}$ where 
 $$\xi_1: 0 \rightarrow Y \rightarrow U \rightarrow Z \rightarrow 0 \quad \text{and} \quad \xi_2: 0 \rightarrow Z \rightarrow U' \rightarrow Y \rightarrow 0$$
 are two non-split short exact sequences. 
 \end{enumerate}
\end{proposition}

We need the following lemma, stated without proof in \cite{GeLeSc08}, which can also be seen as a corollary of the much more general \cite[Proposition 12.4]{GeLeSc11}. For the sake of convenience, we give a direct proof.

\begin{lemma} \label{abc}
 For any $j \in J$, at least one of the following complexes is exact:
  \begin{align*} \Hom_\Pi(S_j, F\xi_1):  \quad & 0 \to \Hom_\Pi(S_j, FY) \to \Hom_\Pi(S_j, FU) \to \Hom_\Pi(S_j, FZ) \to 0, \\ \Hom_\Pi(S_j, F\xi_2): \quad  & 0 \to \Hom_\Pi(S_j, FZ) \to \Hom_\Pi(S_j, FU') \to \Hom_\Pi(S_j, FY) \to 0.\end{align*}
\end{lemma}

\begin{proof}
 Applying $F$ to $\xi_1$ and $\xi_2$, we get short exact sequences $F \xi_1$ and $F \xi_2$.
 Applying $\Hom_\Pi(S_j, -)$ to $F \xi_1$ and $F \xi_2$, it is enough to show that at least one of the induced morphisms
 $$\Hom_\Pi(S_j, FZ) \to \Ext^1_\Pi(S_j, FY) \quad \text{and} \quad \Hom_\Pi(S_j, FY) \to \Ext^1_\Pi(S_j, FZ)$$
 vanishes. Without loss of generality, suppose that there exists $f: S_j \hookrightarrow FZ$ such that the induced extension in $\Ext^1_\Pi(S_j, FY)$ is non-zero. We deduce that
 $$\Ext^1_\Pi(f, FY): \Ext^1_\Pi(FZ, FY) \to \Ext^1_\Pi(S_j, FY)$$
 is non-zero, so injective as $\dim \Ext^1_\Pi(FZ, FY) = 1$. As $\Pi$ is stably $2$-Calabi-Yau, we get that
 $$\Ext^1_\Pi(FY, f): \Ext^1_\Pi(FY, S_j) \to \Ext^1_\Pi(FY, FZ)$$
 is surjective, so there is a pushout diagram
 $$\xymatrix{
  0 \ar[r] & S_j \ar[d]^f \ar[r] & M \ar[r] \ar[d] & FY \ar[r] \ar@{=}[d] & 0 \\
  0 \ar[r] & FZ \ar[r] & FU' \ar[r] & FY \ar[r] & 0
 }$$
 the second row of which is the image by $F$ of the short exact sequence of Proposition \ref{mutG} (b). So, as $\Ext^1_\Pi(S_j, S_j) = 0$, any $g: S_j \to FY$ factors through $M$, hence through $FU'$. Therefore, the map $\Hom_\Pi(S_j, FY) \to \Ext^1_\Pi(S_j, FZ)$ vanishes.
\end{proof}

\begin{proof}[Proof of proposition \ref{mutG}]
 (a) It is an obvious consequence of the property for $\tilde \varphi$ and our definition of $\psi$.

 (b) Consider decompositions $U \cong U_0 \oplus U_1$ and $U' \cong U'_0 \oplus U'_1$ where $U_1$ and $U'_1$ are maximal direct summands contained in $\add Ae$. Thanks to Proposition \ref{descexseq} (b), we have
 $$U_1 = \bigoplus_{j \in J} (A e_j)^{a_j+b_j-c_j} \quad \text{and} \quad U'_1 = \bigoplus_{j \in J} (A e_j)^{a_j+b_j-c'_j}$$
 where, for $j \in J$,
 \begin{itemize}
  \item $a_j = \dim \Hom_{\Pi/I_J}(S_j, FY) = \dim \Hom_\Pi(S_j, FY) $;
  \item $b_j = \dim \Hom_{\Pi/I_J}(S_j, FZ) = \dim \Hom_\Pi(S_j, FZ)$;
  \item $c_i = \dim \Hom_{\Pi/I_J}(S_j, FU) = \dim \Hom_\Pi(S_j, FU)$;
  \item $c'_i = \dim \Hom_{\Pi/I_J}(S_j, FU') = \dim \Hom_\Pi(S_j, FU')$.
 \end{itemize}

  By Lemma \ref{abc}, using $\alpha_j$'s and $\beta_j$'s of Theorem \ref{mutGLS}, we have $a_j+b_j - c_j = \max(0, c'_j - c_j) = \alpha_j$ and $a_j+b_j-c'_j = \beta_j$. Thus, Theorem \ref{mutGLS} implies
 \begin{align*}
  \psi_Y \psi_Z &= \tilde \varphi_{FY} \tilde \varphi_{FZ} = \tilde \varphi_{FU} \prod_{j \in J} \Delta_j^{\alpha_j} + \tilde \varphi_{FU'} \prod_{j \in J} \Delta_j^{\beta_j} = \psi_{U_0} \psi_{U_1} + \psi_{U'_0} \psi_{U'_1} = \psi_U + \psi_{U'} \qedhere.
 \end{align*}
\end{proof}

Now, we can deduce the proof of Theorem \ref{categG}:

\begin{proof}[Proof of Theorem \ref{categG}]
 Using Theorem \ref{categGLS}, it is enough to observe that $F: \CM_e A \to \Sub U$ induces a bijection between isomorphism classes of basic cluster tilting objects. This is immediate as $F$ induces a triangle equivalence $\underline{\CM}_e A \cong \underline{\Sub} U$. More precisely, basic cluster tilting objects of $\CM_e A$ are of the form $Ae \oplus T$ where $T$ has no direct summand in $\add Ae$, and the indecomposable direct summands of $T$ correspond bijectively to the indecomposable direct summands of $FT$.
\end{proof}

\bibliographystyle{alphanum}
\bibliography{../../biblio/biblio}

\end{document}